\PassOptionsToPackage{dvipsnames}{xcolor}
\pdfoutput=1 
\documentclass[leqno]{siamart1116}
\usepackage{amssymb}
\usepackage{graphicx}
\usepackage{xspace}
\usepackage{multirow}
\usepackage{siunitx}
\newcommand{\mynum}[2]{{\qty[scientific-notation=false, round-mode=figures,round-precision = 5, drop-zero-decimal, round-pad = false]{#1}{#2}}}
\usepackage{dutchcal}
\usepackage{derivative}
\usepackage{bm,bbm}
\usepackage{./Macros/mydef}
\usepackage[numbers,sort,compress]{natbib}
\let\cite=\citet
\usepackage{enumerate}
\usepackage{dsfont}
\usepackage{enumitem}
\usepackage{capt-of}
\usepackage[algo2e,linesnumbered]{algorithm2e}
\usepackage{algorithmic}
\usepackage{caption}
\usepackage{subcaption}
\usepackage{booktabs}
\usepackage{placeins}

\usepackage{tikz}
\usetikzlibrary{arrows,decorations,shapes, patterns}
\usepackage[all]{xy}
\usepackage{xfrac}

\begin{document}

\newcommand\footnotemarkfromtitle[1]{%
  \renewcommand{\thefootnote}{\fnsymbol{footnote}}%
  \footnotemark[#1]%
  \renewcommand{\thefootnote}{\arabic{footnote}}}

\newcommand{\TheTitle}{A high-order explicit Runge-Kutta approximation technique for the Shallow Water Equations}
\newcommand{\TheAuthors}{J.-L. Guermond, M. Maier, E.~J. Tovar}

\headers{Well-balanced invariant-domain-preserving approximation}{\TheAuthors}

\title{{\TheTitle}\thanks{Draft version, March 25, 2024%
  \funding{This material is based upon work supported in part by the
    National Science Foundation grants DMS-1619892 and DMS-2110868 (JLG),
    DMS-1912847 and DMS-2045636 (MM), by the Air Force Office of Scientific
    Research, under grant/contract number FA9550-23-1-0007 (JLG, MM), the
    Army Research Office, under grant number W911NF-19-1-0431 (JLG), the U.S.
    Department of Energy by Lawrence Livermore National Laboratory under
    Contracts B640889, B641173 (JLG). ET acknowledges the support from the
    U.S. Department of Energy’s Office of Applied Scientific Computing
    Research (ASCR) and Center for Nonlinear Studies (CNLS) at Los Alamos
    National Laboratory (LANL) under the Mark Kac Postdoctoral Fellowship in
    Applied Mathematics.}}}

\author{Jean-Luc Guermond\footnotemark[1]
  \and Matthias Maier\footnotemark[1]
  \and Eric J. Tovar\footnotemark[2]
}

\maketitle

\renewcommand{\thefootnote}{\fnsymbol{footnote}}
\footnotetext[1]{Department of Mathematics, Texas A\&M University 3368
  TAMU, College Station, TX 77843.}
\footnotetext[2]{Theoretical Division, Los Alamos National Laboratory, P.O. Box 1663, Los Alamos, NM, 87545.}

\renewcommand{\thefootnote}{\arabic{footnote}}
\begin{abstract}
  We introduce a high-order space-time approximation of the Shallow
  Water Equations with sources that is invariant-domain preserving
  (IDP) and well-balanced with respect to rest states.  The employed
  time-stepping technique is a novel explicit Runge-Kutta (ERK)
  approach which is an extension of the class of ERK-IDP methods
  introduced by Ern and Guermond (SIAM J. Sci. Comput. 44(5),
  A3366--A3392, 2022) for systems of non-linear conservation
  equations.
  The resulting method is then numerically illustrated through
  verification and validation.
\end{abstract}





\section{Introduction}\label{sec:introduction}
The development of robust, accurate and efficient discretization techniques
for the shallow water equations (and variations thereof) is an important
task for the field of geosciences. These models are widely used for
applications in coastal hydraulics, in-land flooding and climate
prediction. Using numerical methods that are accurate in the temporal and
spatial domain is desirable for such applications as long-time simulations
spanning longer than 24 hours are common. Robustness is also an important
criterion for such methods. Here, we say that a method is robust if it can
handle dry states and if it can preserve important equilibrium states which
could be either the rest state~\citep{bermudez_vazquez_94,
  greenberg_leroux_96} or time-independent solutions with nonzero
velocity~\citep{Noelle_Xing_Shu_2007, Shu_Xing_2014,
  Chertock_Kurganov_2015}. Numerical methods that preserve such equilibrium
states are said to be well-balanced. The reader is referred to the book
of~\cite{bouchut_2004} for a review on issues related to well-balancing.
Finally, to be useful for practitioners, numerical methods for solving the
shallow water equations must be versatile; in particular, they must be able to
handle unstructured meshes. But, achieving robustness with respect to dry states,
well-balancing, and high-order accuracy in space and time
on unstructured meshes is challenging. The task becomes even more
complex when external source terms besides topography are added.

Developing well-balanced methods that are robust with respect to dry
states is an active topic of research; see~\citep{audusse_etal_2004, Bollermann_2011,
  Gallardo2007574, Kurganov_Petrova_2007, Perthame_Simeoni_2001,
  Ricchiuto_Bollermann_2009}.  Recent work on high-order schemes for
the shallow water equations without external source terms that are
well-balanced and robust with respect to dry states has been proposed
in~\citep{castro2019third} for finite volumes +
central WENO on structured meshes and in~\citep{hajduk2022bound}
for continuous finite elements on
unstructured meshes.
The reader is also referred to~\citep{liang2009numerical,
  Duran_Marche_Turpault_Berthon_2015,
  Chertock_Kurganov_2015,Guermond_2018_SISC}
for other recent works that consider the
inclusion of external source terms such as friction and rain effects.
With the advancement of
computing architectures, there has also been some development on
efficient implementation of numerical methods for the shallow water
equations; see \eg~\citep{brodtkorb2012efficient,Dietrich_etal_2011,delmas2022multi,
  caviedes2023serghei}.

The goal of this work is to present an explicit approximation of the
shallow water equations with topography and external sources that is
well-balanced and high-order accurate in space and time.  Our
theoretical and algorithmic work is supplemented with a high
performance implementation suitable for high fidelity simulations that
is made freely available as part of the \texttt{ryujin}
project~\citep{ryujin-2021-1,ryujin-2021-3}.\footnote{\url{https://github.com/conservation-laws/ryujin}}
The purpose of this work is to provide a stepping stone for
various multi-physics extensions of the shallow water equations
that require an implicit-explicit (IMEX) time
discretization. Such variations include the Serre-Green-Naghdi
Equations \citep{green1974theory, serre_1953} for dispersive water
waves and the coupling of the shallow water equations to subsurface
models such as Richard's equation.
The starting point for this work are the
approximation techniques introduced in~\citep{azerad_2017,
  Guermond_2018_SISC} for the shallow water equations. Unfortunately,
the methodology discussed in~\citep{azerad_2017, Guermond_2018_SISC}
has two drawbacks: \textup{(i)} the high-order spatial approximation
is not fully well-balanced when a shoreline is present; that is to
say, the shoreline must coincide with the mesh for the method to
maintain well-balancing. \textup{(ii)} the time-discretization is
limited in accuracy and efficiency due to the use of explicit
\emph{strong stability preserving} (SSP) explicit Runge-Kutta (ERK)
methods which are known to be limited to fourth-order accuracy
(see~\cite[Thm.~4.1]{ruuth2002two}) and generally have an
\emph{efficiency ratio} significantly smaller than
one~\citep[Def.\,1.1]{Ern_Guermond_2022}.  Here, we provide solutions
to these drawbacks. Specifically, we revisit the low-order method
proposed in \citep[\S3]{azerad_2017} and construct a high-oder version
thereof that is unconditionally well-balanced and more robust with
respect to dry states than the ones outlined
in~\citep[\S4]{azerad_2017}
and~\citep[\S5\&6]{Guermond_2018_SISC}. The novelty of the spatial
discretization introduced in this work is threefold: \textup{(i)} we
introduce modified auxiliary states (see
Eq.~\eqref{eq:auxiliary_states}) that act as local Riemann averages
for hydrostatic reconstructed left/right states; \textup{(ii)} we
rewrite the low-order method as a convex combination of these
auxiliary states and external source terms (see
Lemma~\ref{Lem:low_order_stability}); \textup{(iii)} we introduce
novel local bounds in space and time that control the updated velocity
thereby avoiding blow-up and unnecessary  time-step restrictions (see
Lemma~\ref{Lem:local_bounds}). These modifications allow the final
limited updated to be high-order accurate in space and time,
invariant-domain preserving and well-balanced \wrt~rest states
without any restrictions on the underlying mesh.

The paper is organized as follows. We briefly present the mathematical
model and relevant properties in Section~\ref{sec:model_problem}. In
Section~\ref{sec:forward_euler}, we introduce a discretization-independent
high-order spatial approximation to the shallow water equations with
forward Euler time stepping and a convex limiting procedure. The main
results of this section are Lemma~\ref{Lem:low_order_stability},
Proposition~\ref{prop:entropy_inequality}, Proposition~\ref{prop:IDP} and
Proposition~\ref{prop:well-balancing}. Then, using the convex limiting
methodology used for high-order spatial discretization, we introduce in
Section~\ref{sec:erk_IDP} a high-order in time invariant-domain
preserving (IDP) explicit Runge-Kutta (ERK) method. Finally, we verify and
validate the numerical method in Section~\ref{sec:illustrations}. For the
sake of completeness, we detail the implementation of boundary conditions
for our method in Appendix~\ref{sec:boundary_conditions}.


\section{The model problem}\label{sec:model_problem}
Let $D$ be a polygonal domain in $\polR^d$, $d \in \{1,2\}$, occupied by a
body of water whose evolution in time under the action of gravity is
modeled by the shallow water equations (also known as the Saint-Venant equations). Let $\bx\in D$ be the position
vector and $t > 0$ be the time variable. Let $\bu\eqq(\waterh,
  \bq)\tr\in\Real^{d+1}$ be the dependent variable of the system where
$\waterh(\bx, t)$ is the water depth and $\bq(\bx, t)\in\Real^d$ is the
depth-averaged momentum vector of the fluid, also called discharge. Let
$z(\bx)$ be the known topography mapping.
We henceforth assume that $z$ is in $W^{1,\infty}(\Dom;\Real)$ to
make sure that $\GRAD z$ is a bounded function and thereby avoiding the
need of properly defining $\waterh\GRAD z$ when $z$ is discontinuous. The
goal of this work is to solve the following system of partial differential
equations in the weak sense:
\begin{subequations}\label{swes}
  \begin{align}
    \begin{aligned}
       & \partial_t \bu +\DIV \polf(\bu) = \polb(\bu,z(\bx)) &
       & \text{for \ae\ }t>0,\,\bx\in D,
      \\
       & \bu(\bx,0) = \bu_0(\bx)                             &   & \text{for \ae\ }\bx\in D,
    \end{aligned}
  \end{align}
\end{subequations}
with $\polf(\bu)\eqq(\bq,\bq\otimes \bv+\frac12g\waterh^2\polI_d)\tr$ and
$\polb(\bu,z(\bx))\eqq (0, - g \waterh\GRAD z)\tr$ where
$\bv\eqq\waterh^{-1}\bq\in\polR^d$ is the (depth-averaged) velocity vector,
$g$ is the gravitational acceleration constant, and
$\polI_d\in\polR^{d\times d}$ is the identity matrix.  For the sake of
completeness, we state a few properties regarding~\eqref{swes} that will be
useful when constructing \emph{physically relevant} approximations at the
discrete level.
\begin{definition}[Invariant set]
  The following convex domain is an \emph{invariant set} (in the sense of
  \citep[Def.\,2.3]{guermond2016invariant}) for the shallow water
  equations~\eqref{swes}:
  \begin{equation}\label{invariant_domain}
    \calA\eqq\big\{\bu=(\waterh,\bq)\tr \in \Real^{d+1}\mid \waterh > 0
    \big\}.
  \end{equation}
\end{definition}
When the fluid is at rest, \ie~$\bq\equiv \bm{0}$, the Shallow Water
Equations~\eqref{swes} reduce to $g\waterh\GRAD(\waterh + z) = \bm{0}$,
which motivates to introduce the following terminology.
\begin{definition}[Problem at rest]\label{def:problem_at_rest}
  A solution $\bu(\bx,t)$ to the shallow water equations is said to
  be \emph{at rest} at time $t$ if
  \begin{align*}
    \begin{aligned}
      \bq(x,t)           & =\bzero        & \; & \text{for \ae\ }\bx\in D,\quad\text{and}
      \\
      (\waterh+z)(\bx,t) & =\text{const.} & \; & \text{\ae\ on connected
        components of }\big\{\bx\in D\;:\; \waterh(\bx,t) > 0\big\}.
    \end{aligned}
  \end{align*}
  \vspace{-1em}
\end{definition}
\begin{definition}[Entropy pairs and entropy solutions]
  The pair $\big(E(\bu), \bF(\bu)\big)$:
  \begin{align}\label{eq:entropy_pair}
     & E(\bu)\eqq \tfrac12g\waterh^2 + \tfrac12\waterh\|\bv\|^2 + g z\waterh, &  &
    \bF(\bu)\eqq\bv\big(E(\bu) + \tfrac12g\waterh^2\big),
  \end{align}
  is an entropy pair for the
  shallow water equations~\eqref{swes} \ie~it satisfies $\DIV\bF(\bu) =
    \big(\GRAD_\bu E(\bu)\big)\tr(\DIV\polf(\bu) - \polb(\bu,z(\bx)))$. We
  call $\bu(\bx,t)$ an \emph{entropy solution} to~\eqref{swes} if it is a
  weak solution to~\eqref{swes} and additionally satisfies the following
  inequality in the weak sense: $\partial_t E(\bu) + \DIV\bF(\bu) \le 0$.
\end{definition}
\begin{remark}[Entropy pair for flat topography]
  When there is no influence due to topography ($z(\bx)\equiv 0$), the
  entropy pair~\eqref{eq:entropy_pair} simplifies to:
  \begin{align}\label{eq:entropy_flat}
     & E_{\text{flat}}(\bu)\eqq \tfrac12g\waterh^2 + \tfrac12\waterh\|\bv\|^2, &  &
    \bF_{\text{flat}}(\bu)\eqq\bv\big(E_{\text{flat}}(\bu) + \tfrac12g\waterh^2\big).
  \end{align}
  and satisfies $\DIV\bF_{\text{flat}}(\bu) = \big(\GRAD_\bu E_{\text{flat}}(\bu)\big)\tr\DIV\polf(\bu)$.
\end{remark}

For various applications, the system~\eqref{swes} is augmented with external
source terms. Some examples include forcing due to bottom friction and
source/sink of the water depth~\citep{Chertock_Kurganov_2015}, Coriolis force~\citep{chertock2018well} and many others. In this
work, we only consider a simple time-independent source due to
rainfall and and the loss of
discharge due to the Gauckler-Manning friction force:
\begin{equation}\label{eq:source_terms}
  \bm{S}(\bu) = (R(\bx), -g n^2\waterh^{-\tfrac43}\bq\|\bv\|_{\ell^2})\tr,
\end{equation}
where $R(\bx) > 0$ and $n$ is
the Gauckler-Manning roughness coefficient.


\section{Well-balanced forward Euler method}\label{sec:forward_euler}
In this section, we introduce a forward Euler approximation to the shallow
water equations that is high-order accurate in space, well-balanced and
invariant-domain preserving. This section lays the foundation for the
high-order explicit Runge-Kutta methodology introduced in
Section~\ref{sec:erk_IDP}.

\subsection{Approximation details}\label{sec:space_approximation}
Let $(0,T)$ be a chosen time interval for~\eqref{swes}.
Let$\{t^n\}_{n\in\intset{0}{N}}$ be a discretization of $(0,T)$ with the
convention that $N\ge 1$, $t^0 = 0$, and $t^{N} = T$. The spatial
approximation is discretization independent and can be either finite
differences, finite volumes, continuous or discontinuous finite elements.
Letting $t^n$ be the current discrete time, we assume that the spatial
approximation of $\bu(\cdot, t^n)$ is entirely defined by the collection of
states $\bsfU^n\eqq\{\bsfU_i^n\}_{i\in\calV}$, where $\calV =
\intset{1}{I}$ is the index set for the spatial degrees of freedom and
$\bsfU_i^n\eqq(\sfH_i^n, \bsfQ_i^n)\tr\in\polR^{d+1}$. Here, $\sfH_i^n$ and
$\bsfQ_i^n$ represent approximations of the water depth and discharge
associated with the $i$-th degree of freedom at time $t^n$. To be able to
refer to the water depth and the discharge of an arbitrary state $\tilde
\bsfU=(\tilde \sfH,\tilde \bsfQ)$ in $\Real^{d+1}$, we introduce the linear
mappings $\sfH: \Real^{d+1}\to \Real$ and $\bsfQ: \Real^{d+1}\to \Real^d$
so that $\sfH(\tilde\bsfU) = \tilde\sfH$ and $\bsfQ(\tilde\bsfU)=
\tilde\bsfQ$. We assume that the topography mapping is approximated by the
collection of states: $\sfZ\eqq\{\sfZ_i\}_{i\in\calV}$. We assume that for
every $i\in\calV$, there exists a subset $\calI(i)\subsetneq\calV$ that
collects the local degrees of freedom that interact with $i$, which we call
stencil at $i$. Let $\calI^{*}(i)\eqq \calI(i){\setminus}\{i\}$. We assume
that the underlying spatial discretization provides the following three
quantities for all $i\in\calV$ and all $j\in\calI(i)$:
\begin{enumerate}[font=\upshape,label=(\roman*)]
  \item
    An invertible low-order mass matrix $\polM_{ij}\upL = m_i\delta_{ij}$
    where $m_i > 0$ is called the mass associated with the $i$-th degree of
    freedom;
  \item
    An invertible, symmetric high-order matrix with entries $\polM_{ij}\upH
    = m_{ij}$ such that: $(\polM\upH\sfX)_i = \ssum m_{ij}\sfX_j$ for all
    $\sfX\in\polR^{I}$;
  \item
    A vector $\cij\in\polR^d$ that approximates the gradient operator:
    $\GRAD \sfX(t^n)\approx\ssum\sfX_j^n\cij$.
\end{enumerate}
The local stencil $\calI(i)$ is more precisely defined by
$(j\notin\calI(i))\implies (\cij = \bm{0}$ and $m_{ij}=0$). We further
assume that $\cij = -\bc_{ji}$ whenever $i$ or $j$ is not a boundary degree
of freedom, and $\ssum\cij = 0$ which is necessary for mass conservation.
The high-order mass matrix is related to the low-order mass matrix through
the relation $m_i = \ssum m_{ij}$ to guarantee that $\polM\upL$ and
$\polM\upH$ carry the same mass: $\sum_{i\in\calV} m_i \bsfU_i^n =
\sum_{i\in\calV}\ssum m_{ij}\bsfU_j^n$. Examples of discretization
techniques satisfying the above assumptions are described
in~\citep{guermond2019invariant}.

\subsection{Well-balancing preliminaries}\label{sec:well_balancing}
Before introducing the low- and high-order spatial approximations, we first
define the discrete velocity and what it means to be well-balanced with
respect to rest states. Given a state $(\sfH_i^n,\bsfQ_i^n)\tr$ with
nonzero water depth, $\sfH_i^n>0$, the velocity is defined to be the ratio
$\bsfQ_i^n/\sfH_i^n$. To be robust with respect to dry states, we adopt the
regularization technique from~\cite{Kurganov_Petrova_2007} that avoids the
division by zero when $\sfH_i^n\rightarrow0$. For this purpose, we
introduce a small dimensionless parameter $\epsilon$ and a characteristic
length scale $\waterh_{\max}$ that scales like the average of the water
depth in the problem, and we set
\begin{equation}
  \bsfV_i^n\eqq\frac{2\sfH_i^n}{(\sfH_i^n)^2+
  \max(\sfH_i^n,\epsilon\waterh_{\max})^2}\bsfQ_i^n.
\end{equation}
Notice that $\bsfV_i^n =\bsfQ_i^n/\sfH_i^n$ when $\sfH_i^n\ge
\epsilon\waterh_{\max}$; that is, the regularization is active only when
$\sfH_i^n\le \epsilon\waterh_{\max}$. All the numerical simulations
reported in the paper are done with $\epsilon=\epssw$ and using double
precision floating point arithmetic. The well-balancing techniques in this
work adopt the methodology proposed
in~\cite{audusse_etal_2004,Audusse_Bristeau_2005} known as
\emph{hydrostatic reconstruction} of the water depth.
\begin{definition}[Hydrostatic reconstruction]
  For $i\in\calV$ and $j\in\calI(i)$, the hydrostatic reconstruction
  between $i$ and $j$ of the state $\bsfU_{i}$ and the associated water depth $\sfH_{i}$
 are defined as follows:
  \begin{subequations}
    \begin{align}
      \sfH_{i}^{j,*} & \eqq \max(0,\sfH_i+\sfZ_i-\max(\sfZ_i,\sfZ_j)),
      \label{def_of_Hstar}
      \\
      \bsfU_i^{j,*}  & \eqq \begin{pmatrix}\sfH_{i}^{*,j} \\\bsfV_i
      \sfH_{i}^{j,*}\end{pmatrix}.
      \label{def_of_Ustar}
    \end{align}
  \end{subequations}
\end{definition}
\begin{definition}[Discrete states at rest]
  \label{Def:Rest_at_large}
  A given set of discrete numerical states $\{(\sfH_i^n,
  \bsfQ_i^n)\}_{i\in\calV}$ is said \emph{to be at rest} if the approximate
  momentum $\bsfQ_i^n$ is zero for all $i\in\calV$, and if the approximate
  water depth $\sfH_i^n$ and the approximate bathymetry map $\sfZ_i$ satisfy
  the following property for all $i\in\calV$: $\Histar = \Hjstar$ for all
  $j\in\calI(i)$.
\end{definition}
\begin{remark}
  Note that we have adopted the condition $\Histar = \Hjstar$ to be
  the discrete analog to $(\waterh + z)(\bx, t) = const.$ in
  Def.~\ref{def:problem_at_rest} instead of the natural looking
  identity $\sfH_i\upn + \sfZ_i = \sfH_j + \sfZ_j$. The reason behind
  this choice is the fact that Def.~\ref{Def:Rest_at_large} does not
  need to distinguish between dry and wet states. For example, assume
  that $\sfH_j = 0$ and $\sfZ_i >\sfZ_j$ for all $j\in\calI(i)$.
  Then, $\Histar = \max(0, 0 + \sfZ_i - \sfZ_i) = 0$ and
  $\Hjstar = \max(0, 0 + \sfZ_j - \sfZ_i) = 0$, which gives
  $\Histar = \Hjstar$. However, the condition
  $\sfH_i\upn + \sfZ_i = \sfH_j + \sfZ_j$ breaks down in this case since
  $\sfH_j = 0$ for all $j\in\calI(i)$ would imply that the topography
  mapping should be constant which is not the case.
\end{remark}

\subsection{Low-order spatial approximation}\label{sec:low_order}
We now discuss a low-order method that will serve as safeguard to the
high-order method. We essentially follow~\cite[Sec.~3]{azerad_2017}
which introduced a formally first-order consistent approximation of
the shallow water equations when the spatial discretization consists
of continuous, linear finite elements. But departing from
\cite{azerad_2017} we introduce a different definition of
  the \emph{auxiliary states} (see Eq.~\ref{eq:auxiliary_states}) and
  introduce a different convex combination of theses auxiliary states
  to reconstruct the low order update; see
Lemma~\ref{Lem:low_order_stability}.

Let $t^n$  be the current time and let $\dt\eqq t^{n+1}-t^n$ denote the
current time step. The low-order approximation with forward Euler
time-stepping is then given by:
\begin{subequations}\label{low_order}
  \begin{align}
    \label{eq:low_order_scheme}
    & \frac{m_i}{\tau}(\bsfU_i\upLnp - \bsfU_i^n) = \ssum\FL_{ij},
    \\
    \label{eq:low_order_flux}
    &\begin{aligned}
      \FL_{ij} \eqq{} & -\left(\!\Ujstar\otimes \bsfV_j^n+
      \Uistar\otimes\bsfV_i^n\!\right)\cij + d_{ij}\upLn\!\left(\!\Ujstar -
      \Uistar\!\right)
      \\
      & - \begin{pmatrix}
        0
        \\
        g \cij \left(\frac12(\Hjstar)^2 - \frac12(\Histar)^2+
        (\sfH_i^n)^2\right)
      \end{pmatrix}.
    \end{aligned}
  \end{align}
\end{subequations}
Here, $d_{ij}\upLn\geq0$ is the graph-viscosity coefficient that makes the
method invariant-domain-preserving:
\begin{equation}
    d_{ij}\upLn\eqq\max\Big(\lambda_{\max}(\Uistar, \Ujstar, \bn_{ij})\|\cij\|_{\ell^2},
    \lambda_{\max}(\Ujstar, \Uistar, \bn_{ji})\|\cji\|_{\ell^2},\Big),
\end{equation}
where $\bn_{ij}\eqq \bc_{ij}/\|\bc_{ij}\|_{\ell^2}$ and
$\lambda_{\max}(\Uistar, \Ujstar, \bn_{ij})$ is a guaranteed upper bound on
the maximum wave speed in the local Riemann problem with left and right
states $(\Uistar,\Ujstar)$ and flux $\polf(\cdot)\bn_{ij}$. Analytical expressions for $\lambda_{\max}$ are
given in~\citep[Lem.~3.8]{azerad_2017}. Note that $d_{ij}\upLn =
d_{ji}\upLn$ which is necessary for mass conservation. By convention, we
set $d_{ii}\upLn\eqq-\snsum d_{ij}\upLn$. Observe that $\bsfF_{ij}\upLn =
-\bsfF_{ji}\upLn$ when the topography is flat.
\begin{lemma}[Well-balancing and conservation]
    \label{lem:low-order-WB}
    The scheme $\bsfU^n\mapsto \bsfU\upLnp$ defined in~\eqref{low_order} is
    mass-conservative and well-balanced.
\end{lemma}
\begin{proof}
    See~\cite[Prop.~3.9]{azerad_2017}.
\end{proof}

We now introduce auxiliary states that are meant to be thought of as
averages of the self-similar solution to the local Riemann problem for
the pair $(\Uistar,\Ujstar)$ in the direction $\bn_{ij}$. We do this
to rewrite the low-order scheme~\eqref{low_order} as a convex
combination of these auxiliary states and source terms to extract
local bounds in space and time. These bounds are needed for the convex
limiting procedure described in Section~\ref{sec:convex_limiting}. The
importance of such auxiliary states for nonlinear conservation laws
has been established in the literature and we refer the reader
to~\cite{harten1983upstream, nessyahu1990non} and the references
therein.

Let $i\in\calV$. For every $j\in\calI(i)$, we define the following
auxiliary states for the pair $(i, j)$ as follows:
\begin{equation}
  \label{eq:auxiliary_states}
  \Ubar_{ij}^n\eqq \frac12\big\{\Uistar + \Ujstar\big\}
  - \frac{1}{2 d_{ij}\upLn}\big\{\polf(\Ujstar) - \polf(\Uistar)\big\}\cij,
\end{equation}
with the convention that $\Ubar_{ii}^n\eqq\bsfU_i^n$.  We recall
  that $\Ubar_{ij}^n$ coincides with the exact space average over $[-\frac12,\frac12]$ at time
  $\frac{\|\bc_{ij}\|_{\ell^2}}{2 d_{ij}\upLn}$ of the solution to the
Riemann problem with left and right states $(\Uistar,\Ujstar)$ and
with flux $\polf(v)\bn_{ij}$ provided the viscosity $ d_{ij}\upLn$ is
large enough so that $\lambda_{\max}(\Uistar,
\Ujstar, \bn_{ij})\|\cij\|_{\ell^2},\le d_{ij}\upLn$. Note that the bar states here differ from those
in~\citep[Prop.~3.11]{azerad_2017}. We also define an affine shift
needed for the convex combination update:
\begin{equation}\label{eq:affine}
  \BL_i\eqq\ssum\BL_{ij},\quad \BL_{ij}\eqq-2( d_{ij}\upLn +
  \bsfV_i^n\SCAL \bc_{ij})(\Uistar -\bsfU_i^n).
\end{equation}
\begin{lemma}[Stability]
  \label{Lem:low_order_stability}
  Assume that $\bsfU_i^n\in\calA$ for all $i\in\calV$. Assume also that the
  time step satisfies the restriction $\dt \le
  \min_{i\in\calV}\frac{m_i}{2|d_{ii}\upLn|}$. Then,
  \begin{enumerate}[font=\upshape,label=(\roman*)]
    \item
      \label{Item1:Lem:low_order_stability}
      The following convex combination holds:
      \begin{equation}
        \label{eq1:Lem:low_order_stability}
        \bsfU_i\upLnp =  \Big(1 + \frac{2\tau
        d_{ii}\upLn}{m_i}\Big)(\bsfU_i^n + \frac{\tau}{m_i}\BL_i) + \snsum
        \frac{2\tau d_{ij}\upLn}{m_i}(\Ubar_{ij}^n +
        \frac{\tau}{m_i}\BL_i).
      \end{equation}
    \item
      \label{Item2:Lem:low_order_stability}
      The water depth of $\bsfU\upLnp$ is positive, \ie
      $\bsfU_i\upLnp\in\calA$ for all $i\in\calV$.
    \end{enumerate}
\end{lemma}
\begin{proof}
  In order to show \ref{Item1:Lem:low_order_stability}
  we add and subtract $\BL_i$
  in~\eqref{low_order} and then rearrange:
  \begin{align*}
    \frac{m_i}{\tau}(\bsfU_i\upLnp - \bsfU_i^n)
    & =  \BL_i+\ssum\left[\FL_{ij} - \BL_{ij}\right].
  \end{align*}
  Recalling that $\polf(\bu)\eqq(\bq,\bq\otimes
  \bv+\frac12g\waterh^2\polI_d)\tr$ and $\ssum\bc_{ij}=\bzero$, we have
  \begin{align*}
    \ssum \FL_{ij}  - \BL_{ij}
    & = \ssum -\{\polf(\Ujstar) - \polf(\Uistar)\}\cij + \dij\upLn\{\Uistar
    + \Ujstar\},
    \\
    & =  \ssum 2 \dij\upLn \Ubar_{ij}^n.
  \end{align*}
  Then, recalling that $\Ubar_{ii}^n=\bsfU_i^n$, and $\ssum \dij\upLn=0$
  holds true by virtue of the definition of $d_{ii}\upLn$, we infer that
  \begin{align*}
    \bsfU_i\upLnp & = \bsfU_i^n + \frac{\tau}{m_i}\BL_i + \ssum \frac{2\tau
    \dij\upLn}{m_i} \Ubar_{ij}^n,
    \\
    & =  \bsfU_i^n + \frac{\tau}{m_i}\BL_i + \frac{2\tau d_{ii}\upLn}{m_i}
    \bsfU_{i}^n + \snsum \frac{2\tau \dij\upLn}{m_i} \Ubar_{ij}^n,
    \\
    & =  \bsfU_i^n + \frac{\tau}{m_i}\BL_i + \frac{2\tau d_{ii}\upLn}{m_i}
    (\bsfU_{i}^n + \frac{\tau}{m_i}\BL_i) + \snsum \frac{2\tau
    \dij\upLn}{m_i} (\Ubar_{ij}^n + \frac{\tau}{m_i}\BL_i),
    \\
    & =\Big(1 + \frac{2\tau d_{ii}\upLn}{m_i}\Big)(\bsfU_i^n +
    \frac{\tau}{m_i}\BL_i) + \snsum \frac{2\tau
    d_{ij}\upLn}{m_i}(\Ubar_{ij}^n + \frac{\tau}{m_i}\BL_i).
  \end{align*}
  The above decomposition is a genuine convex combination under
  the CFL condition $1 + \frac{2\tau |d_{ii}\upLn|}{m_i}\geq 0$.

  For \ref{Item2:Lem:low_order_stability} we recall that the water depth
  of the sate $\bsfB_i\upLn$ is given by
  \begin{align*}
    \sfH(\bsfB_i\upLn) \eqq \ssum 2( d_{ij}\upLn + \bsfV_i^n\SCAL
    \bc_{ij})(\sfH_i^n - \Histar).
  \end{align*}
  It is established in~\citep[Prop.~3.7]{azerad_2017} that
  $d_{ij}\upLn+\bsfV_i^n\SCAL\bc_{ij}\ge 0$ and by definition we have
  $(\sfH_i^n - \Histar)\ge 0$. Hence, $\sfH(\bsfB_i\upLn) > 0$. Moreover,
  the definition of $d_{ij}\upLn$ implies that $\sfH(\Ubar_{ij}^n)\ge 0$
  for all $j\ne i\in\calI^*(i)$. The assertion now
  follows directly from the combination
  \eqref{eq1:Lem:low_order_stability} which is convex provided the CFL time step
  restriction holds true.
\end{proof}

\begin{remark}[Source terms]
  In order to incorporate external source terms into the low-order scheme
  we augment \eqref{eq:low_order_scheme} with a suitable low-order
  approximation of $\bS(\bu)$ given by~\eqref{eq:source_terms}:
  \begin{align}
    \notag
    \frac{m_i}{\tau} &(\bsfU_i\upLnp - \bsfU_i^n) = \ssum\FL_{ij}
    + m_i \bS_i^n,
    \quad\text{with}
    \\
    \label{eq:source_terms_discrete}
    \bS_i^n  \eqq{}& \left(R(\ba_i), -g n^2 (\mathcal{H}_i^n)^{-1}
    \bQ_i^n \|\bsfV_i\|_{\ell^2} \right)\tr,
    \\
    &\text{where\ }
    \mathcal{H}_i^n \eqq \frac12\left[(\sfH_i^n)^{4/3} +
    \max((\sfH_i^n)^{4/3}, 2gn^2\tau\|\bsfV_i^n\|_{\ell^2})\right],\notag 
  \end{align}
  is introduced for regularizing the term $\waterh^{-4/3}$ as in \citep[Eq.~(3.3)]{Guermond_2018_SISC}. Here,
  $\ba_i$ denotes a collocation point associated with the $i$th degree
  of freedom. Note that the results in Lemma~\ref{Lem:low_order_stability}
  still hold, provided that one replaces $\BL_i$ by a modified affine shift
  $\BL_i + m_i\bS_i^n$ throughout.
\end{remark}

\begin{proposition}[Entropy inequality]\label{prop:entropy_inequality}
  Assume that the time step satisfies the CFL condition
  $\tau\leq\min_{i\in\calV}\frac{m_i}{2\left|d\upLnp_{ii}\right|}$. Then,
  the low-order update satisfies the following discrete entropy inequality
  for every $i\in\calV$:
  \begin{multline}
    \label{eq:entropy_ineq_full}
    \frac{m_i}{\tau}\big(E_{\textup{flat}}(\bsfU_i\upLnp\big) -
    E_{\textup{flat}}(\bsfU_i\upn)\big) +
    \ssum\Big\{\big(\bF_{\textup{flat}}(\Ujstar) -
    \bF_{\textup{flat}}(\Uistar)\big)\cij -
    \\
    d_{ij}\upLn\big(E_{\textup{flat}}(\Ujstar) -
    E_{\textup{flat}}(\Uistar)\big)\Big\}
    \;\leq\; \GRAD_\bu E_{\textup{flat}}(\bsfU_i\upLn)\,\SCAL\,\BL_i
    \;+\;\mathcal{O}(\tau).
  \end{multline}
  When there is no influence due to topography, the entropy inequality
  reduces to:
  \begin{multline}
    \label{eq:entropy_ineq_flat}
    \frac{m_i}{\tau}\left(E_{\textup{flat}}(\bsfU_i\upLnp) - E_{\textup{flat}}(\bsfU_i\upn)\right)
    \\
    + \ssum\left\{\left(\bF_{\textup{flat}}(\bsfU_j) -
    \bF_{\textup{flat}}(\bsfU_i)\right)\cij -
    d_{ij}\upLn\left(E_{\textup{flat}}(\bsfU_j) -
    E_{\textup{flat}}(\bsfU_i)\right)\right\}\leq0
  \end{multline}
\end{proposition}
\begin{proof}
  We first rewrite the convex combination~\eqref{eq1:Lem:low_order_stability}
  as follows:
  \begin{equation*}
    \bsfU_i\upLnp - \frac{\tau}{m_i}\BL_i
    = \Big(1 + \frac{2\tau d_{ii}\upLn}{m_i}\Big)
    \bsfU_i^n + \snsum \frac{2\tau \dij\upLn}{m_i} \Ubar_{ij}^n,
  \end{equation*}
  and then observe that
  \begin{multline}
    \label{eq:magic}
    E_{\textup{flat}}(\Ubar_{ij}^n) \leq \frac12\left(E_{\textup{flat}}(\Uistar)
    + E_{\textup{flat}}(\Ujstar)\right)
    \\
    -\;\frac{1}{2d_{ij}\upLn}\left(\bF_{\textup{flat}}(\Ujstar) -
    \bF_{\textup{flat}}(\Uistar)\right)\cij
  \end{multline}
  holds true for the auxiliary states. Combining these, as well as
  exploiting the convexity of $E_{\textup{flat}}$ gives rise to an entropy inequality:
  \begin{multline}
    \label{eq:entropy_ineq_intermediate}
    \frac{m_i}{\tau}\Big(E_{\textup{flat}}\big(\bsfU_i\upLnp -
    \frac{\tau}{m_i}\BL_i\big) - E_{\textup{flat}}(\bsfU_i\upn)\Big)
    \\
    + \ssum\Big\{\big(\bF_{\textup{flat}}(\Ujstar) -
    \bF_{\textup{flat}}(\Uistar)\big)\cij -
    \\
    d_{ij}\upLn\big(E_{\textup{flat}}(\Ujstar) -
    E_{\textup{flat}}(\Uistar)\big)\Big\}
    \;\leq\;0
  \end{multline}
  As a last ingredient for showing \eqref{eq:entropy_ineq_full} we use the
  following Taylor series expansion:
  \begin{align*}
    E_{\textup{flat}}\big(\bsfU_i\upLnp - \frac{\tau}{m_i}\BL_i\big)
    \;&=\;
    E_{\textup{flat}}\big(\bsfU_i\upLnp) \,-\, \GRAD_\bu
    E_{\textup{flat}}(\bsfU_i\upLnp)\,\SCAL\,\BL_i \,+\,\mathcal{O}(\tau)
    \\
    \;&=\;
    E_{\textup{flat}}\big(\bsfU_i\upLnp) \,-\, \GRAD_\bu
    E_{\textup{flat}}(\bsfU_i\upLn)\,\SCAL\,\BL_i \,+\,\mathcal{O}(\tau).
  \end{align*}
  Finally,
  \eqref{eq:entropy_ineq_intermediate} readily reduces to
  \eqref{eq:entropy_ineq_flat} for the case of constant bathymetry.
\end{proof}

\subsection{High-order spatial approximation}
\label{sec:high_order}
We now present a high-order spatial approximation of the problem. For
$i\in\calV$ and $j\in\calI(i)$, we define the following high-order flux:
\begin{multline}
    \label{eq:high_order_flux}
    \FH_{ij}\eqq{} -\left(\bsfU_j^n\otimes\bsfV_j^n+\bsfU_i^n\otimes
    \bV_i^n\right)\cij + d_{ij}\upHn\left(\Ujstar - \Uistar\right)
    \\
    - \begin{pmatrix}
      0\\g(\sfH_i^n \sfH_j^n + \sfH_i^n(\sfZ_j - \sfZ_i))\bc_{ij}
    \end{pmatrix},
\end{multline}
where $d_{ij}\upHn = d_{ji}\upHn$ is a high-order graph viscosity. The
symmetry of $d_{ij}\upHn$ implies that $\bsfF_{ij}\upH=-\bsfF_{ji}\upH$
holds true when the topography is flat. We set $\bsfF_i\upH\eqq
\ssum\bsfF_{ij}\upH$. The high-order graph viscosity
coefficient $d_{ij}\upHn$ is defined as follows:
\begin{equation*}
  d_{ij}\upHn\eqq d_{ij}\upLn\frac{\alpha_i^n +
  \alpha_j^n}{2}\quad\text{for }i\neq j,\qquad d_{ii}\upHn\eqq\snsum
  d_{ij}\upHn.
\end{equation*}
Here, $\alpha_i^n\in[0,1]$ is an \textit{indicator} for \emph{entropy
production} and is defined as follows for each $i\in\calV$:
\begin{align}
  & \alpha_i^n\eqq\frac{\abs{N_i^n}}{D_i^n +\epsilon D_{\max}},
  \\
  & N_i^n\eqq\ssum\big\{\bF_{\textup{flat}}(\bsfU_j^n)-(\GRAD_{\bu}
  E_{\textup{flat}}(\bsfU_i^n))\tr\DIV\polf(\bsfU_j^n) \big\}\cij,
  \\
  & D_i^n\eqq \abs{\ssum\bF_{\textup{flat}}(\bsfU_j^n)\cij} +
  \abs{\ssum(\GRAD_{\bu}
  E_{\textup{flat}}(\bsfU_i^n))\tr\DIV\polf(\bsfU_j^n)\cij}.
\end{align}
The small number $\epsilon D_{\textup{max}}\eqq\epsilon\times \sqrt{g
\waterh_{\textup{max}}}\tfrac12g\waterh_{\textup{max}}^2$ is meant to avoid
division by zero when either the water depth or velocity is zero; or the
entropy is constant.

As the high-order approximation requires estimating the inverse of the
high-order mass matrix to reduce the dispersion effects, we proceed as in
\citep[\S3.4]{Guermond_Nazarov_Popov_Yang_2014} to approximate
$(\polM\upH)^{-1}$. Using the expression $(\polM\upH)^{-1} = (\polM\upL)^{-1}
\big(\polI - (\polM\upL - \polM\upH) (\polM\upL)^{-1}\big)^{-1}$ and
setting $\polB\eqq (\polM\upL - \polM\upH) (\polM\upL)^{-1}$, we
approximate $(\polM\upH)^{-1}$ by $(\polM\upL)^{-1}(\polI+\polB)$. This
expansion is shown in \citep[Prop.~3.1]{Guermond_Pasq_2013} to be
superconvergent and to remove the dispersion error for the approximation of
the linear transport equation with piecewise linear continuous finite
elements on uniform meshes. Let $\{b_{ij}\}_{j\in\calI(i)}$ for all
$i\in\calV$ be the entries of $\polB$,~\ie $b_{ij} = \delta_{ij} -
\frac{m_{ij}}{m_j}$ and $b_{ji} = \delta_{ij} - \frac{m_{ji}}{m_i}$. Then,
recalling that $\ssum b_{ji}=0$, the provisional high-order approximation
with forward Euler time-stepping is given by:
\begin{equation}\label{eq:high_order}
    \frac{m_{i}}{\tau}(\bsfU_i\upHnp - \bsfU_i^n) =  \ssum
    \FH_{ij}  + b_{ij}\FH_j - b_{ji}\FH_i .
\end{equation}
As the provisional high-order update $\bsfU_i\upHnp$ defined above is not
well-balanced and also not guaranteed to be invariant-domain-preserving. We
present in the next section a convex limiting technique that combines the
low-order update and the provisional high-order update to make the final
update at $t^{n+1}$ high-order accurate, well-balanced, and
invariant-domain preserving.

\begin{remark}[Spatial accuracy]
  The spatial accuracy of the method~\eqref{eq:high_order} delivers close
  to optimal accuracy for smooth solutions for the underlying spatial
  discretization. For example, assuming the discretization is based on
  continuous linear finite elements, the method is formally second-order
  accurate.
\end{remark}

\subsection{Convex limiting procedure}
\label{sec:convex_limiting}
We now detail the convex limiting procedure. The methodology is loosely
based on~\citep{Guermond_2018_SISC,Guermond_Popov_Tovar_Kees_JCP_2019} and
follows the common FCT ideology
(see:~\cite{zalesak1979fully,boris1997flux,kuzmin2012flux}). The novelty of
the approach proposed in the paper resides in the definition of the local
bounds and the incorporation of sources in the convex combination
introduced in Lemma~\ref{Lem:low_order_stability}.

\subsubsection{Local bounds}
For each $i\in\calV$, we let $\{\lambda_j\}_{j\in\calI^{*}(i)}$ be any set
of positive coefficients that sum up to 1 using the index set $\calI^{*}(i)$.
In the numerical illustrations reported at the end of the paper we use
$\lambda_i = \frac{1}{ \text{Card}(\calI^*(i))}$. Subtracting~\eqref{low_order} from~\eqref{eq:high_order} we obtain
\begin{align}
  \notag
  & (\bsfU_i\upHnp - \bsfU_i\upLnp) = \snsum\lambda_i \Pij^n,
  \\
  \intertext{with}
  \label{eq:P_ij}
  & \Pij^n\eqq \frac{\tau }{m_i \lambda_i}\big\{\FH_{ij}
  - \FL_{ij} + b_{ij}\FH_j - b_{ji}\FH_i\big\}.
\end{align}
Notice that the coefficients $\Pij$ are skew-symmetric.
The key principle of the convex limiting strategy is as follows: For all
$i\in\calV$ and all $j\in\calI^{*}(i)$, we look for a set of symmetric limiting
coefficients $\ell_{ij}^n\in[0,1]$ such that the limited update
$\bsfU_i\upLnp + \snsum\ell_{ij}^n \lambda_i \Pij^n$ satisfies
reasonable properties and is well-balanced. After finding this collection
of limiting coefficients, we define the final update to be
\begin{equation}
  \label{eq:final_update}
  \bsfU_i\upnp = \bsfU_i\upLnp + \snsum\ell_{ij}^n \lambda_i\Pij^n.
\end{equation}
Notice that the final update can be equivalently written as:
\begin{equation*}
  \bsfU_i\upnp = \snsum\lambda_i\left(\bsfU_i\upLnp + \ell_{ij}^n
  \Pij^n\right).
\end{equation*}
which is just a convex combination of limited states. We now explain
how we define
the local bounds which are used to define the limiting
coefficients. The following details differ from our previous
work~\citep{Guermond_2018_SISC,Guermond_Popov_Tovar_Kees_JCP_2019}. Taking
inspiration from the convex combination~\eqref{eq1:Lem:low_order_stability}
for all $i\in\calV$ and all ${j\in\calI(i)}$, we set
$\overline{\bsfW}_{ij}^n\eqq \Ubar_{ij}^n + \frac{\dt}{m_i} \bsfB_i\upLn$.
We then define the minimum and maximum local bound on the water depth by
setting:
\begin{equation}
    \sfH_{i,\min}^n = \min_{j\in\calI(i)} \sfH(\overline{\bsfW}_{ij}^n), \qquad
    \sfH_{i,\max}^n = \max_{j\in\calI(i)} \sfH(\overline{\bsfW}_{ij}^n).
\end{equation}
In order to control the potential blow-up of the velocity in dry states, we
introduce the quantity $\sfH(\bsfU)^2
\sfV^{2,n}_{i,\max}-\|\bsfQ(\bsfU)\|_{\ell^2}^2$ where
\begin{equation}
    \sfV^{2,n}_{i,\max} = \max_{j\in\calI(i)}
    \|\bsfV(\overline{\bsfW}_{ij}^n)\|^2_{\ell^2}.
\end{equation}
Here, $\bsfV(\bsfU)$ is the regularized version of the velocity. In~\citep[p.~A3889]{Guermond_2018_SISC},
the authors propose a limiting technique based on the kinetic energy
($\tfrac12\|\bsfQ(\bsfU)\|_{\ell^2}^2/\sfH(\bsfU)$), but we found that the approach we 
propose here is more robust with respect to dry states.
A control on the velocity via limiting is also adopted
in~\cite[Sec.~3.2]{hajduk2022bound}. 
The following result is essential to
establish the validity of the limiting process.
\begin{lemma}\label{Lem:local_bounds}
  Assume that $\bsfU_i^n\in\calA$ for all $i\in\calV$. Assume also that the
  time step satisfies the restriction $\dt \le
  \min_{i\in\calV}\frac{m_i}{2|d_{ii}\upLn|}$. Then,
  \begin{align}
    \label{H_bound:Lem:local_bounds}
    &\sfH_{i,\min}^n  \le \sfH(\bsfU_i\upLnp) \quad \text{and}\quad 
     \sfH(\bsfU_i\upLnp)\le \sfH_{i,\max}^n,
    \\
    \label{K_bound:Lem:local_bounds}
    &\|\bsfQ(\bsfU_i\upLnp)\|^2_{\ell^2}
     \le \sfH(\bsfU_i\upLnp) \sfV_{i,\max}^{2,n}
     \qquad \textup{if } \epsilon\sfh_{\max}\le \sfH_{i,\min}^n.
  \end{align}
\end{lemma}
\begin{proof}
  The bound~\eqref{H_bound:Lem:local_bounds} is a direct consequence of the
  combination~\eqref{eq1:Lem:low_order_stability} being convex and the mapping
  $\bsfU\to \sfH(\bsfU)$ being linear. We now prove the second bound. We
  first observe that the mapping $\Real_{>0}\CROSS\Real^{d}\ni \bsfU \to
  \|\bsfQ(\bsfU)\|_{\ell^2}^2/\sfH(\bsfU)\in \Real$ is convex. Then, since
  $\bsfU\to \sfH(\bsfU)$ is linear, the mapping
  $\Real_{>0}\CROSS\Real^{d}\ni \bsfU \to
  \|\bsfQ(\bsfU)\|_{\ell^2}^2/\sfH^2(\bsfU)\in \Real$ is quasi-convex due
  to~\citep[Lem.~7.4]{guermond2019invariant}. An application
  of~\citep[Lem.~7.2]{guermond2019invariant} yields:
  $\|\bsfV(\bsfU_i\upLnp)\|^2_{\ell^2}\leq  \sfV^{2,n}_{i,\max}$. As we
  assumed that $\sfH_{i,\min}^n \ge \epsilon\sfh_{\max}$, we infer that
  $\sfH(\overline{\bsfW}_{ij}^n)\ge \epsilon\sfh_{\max}>0$ and
  $\sfH(\bsfU_i\upLn)\ge \epsilon\sfh_{\max}>0$. This implies in particular
  that $\bsfQ(\overline{\bsfW}_{ij}^n) = \sfH(\overline{\bsfW}_{ij}^n)
  \bsfV(\overline{\bsfW}_{ij}^n)$ and $\bsfQ(\bsfU\upLn) =
  \sfH(\bsfU_i\upLn) \bsfV(\bsfU\upLn)$. Hence,
  \begin{equation*}
    \|\bsfV(\bsfU_i\upLnp)\|^2_{\ell^2}\leq  \sfV^{2,n}_{i,\max} \implies
    \|\bsfQ(\bsfU_i\upLnp)\|_{\ell^2}^2 \leq \sfH(\bsfU_i\upLnp)^2
    \sfV^{2,n}_{i,\max}.
  \end{equation*}
  This completes the proof.
\end{proof}

\begin{remark}[Bounds relaxation]
  To achieve optimal accuracy in $L^p$-norms, $p\ge 1$, for smooth solutions, the
  bounds defined above must be relaxed. For the sake of brevity, we refer
  the reader to~\citep[Sec.~7.6]{guermond2019invariant} where this is
  discussed in detail.
\end{remark}

\begin{remark}[Source terms]
  In order to incorporate the external source term $\bS(\bu)$ given
  by~\eqref{eq:source_terms} into the high-order scheme and the subsequent
  convex limiting procedure we change the definition of $\Pij$ as follows:
  \begin{multline*}
    \Pij^n\eqq \frac{\tau }{m_i \lambda_i}\Big\{\FH_{ij} - \FL_{ij} +
    b_{ij}\FH_j - b_{ji}\FH_i
    \\
    +m_{ij}\bS_j^n-m_{ij}\bS_i^n
    +b_{ij}\big(\sum_{k\in\calI(j)}m_{jk}\bS_{k}^n\big)
    -b_{ji}\big(\sum_{k\in\calI(i)}m_{ik}\bS_{k}^n\big)
   \Big\},
  \end{multline*}
  where $\bS_i^n$ is again given by \eqref{eq:source_terms_discrete}.
\end{remark}

\subsubsection{Optimal limiting coefficient}
We now detail the process for finding near optimal limiting coefficients
$l_{ij}$. We introduce the functionals:
$\Psi_1(\bsfU)\eqq\sfH(\bsfU)-\sfH_i^{n,\min},
\,\Psi_2(\bsfU)\eqq\sfH_i^{n,\max}-\sfH(\bsfU),\,
\Psi_3(\bsfU)\eqq\sfH(\bsfU)^2\sfV^{2,n}_{i,\max} -
\|\bsfQ(\bsfU)\|^2_{\ell^2}$. The strategy is as follows: for each
$k\in\{1,2,3\}$, we find $\ell\in[0,1]$ such that
$\Psi_k(\bsfU_i\upLnp+\ell \bsfP_{ij}^n)\geq0$ in a sequential manner.

We first limit the water depth. To ensure robustness with respect to dry
states, we introduce for $i$ in $\calV$:
\begin{equation}\label{eq:lij_h}
  \ell_j^{i,\waterh} =
  \begin{cases}
    \min\Big(\frac{\abs{\sfH_{i,\min}^n -
    \sfH(\bsfU_i\upLnp)}}{\abs{\sfP_{ij}^\waterh} + \epsilon
    \sfH_{i,\max}^n},1\Big), & \text{if }\sfH(\bsfU_i\upLnp) +
    \sfP_{ij}^\waterh < \sfH_{i,\min}^n,
    \\
    1,                                  & \sfH_{i,\min}^n \leq
    \sfH(\bsfU_i\upLnp) + \sfP_{ij}^\waterh  \leq \sfH_{i,\max}^n, \\
    \min\Big(\frac{\abs{\sfH_{i,\max}^n -
    \sfH(\bsfU_i\upLnp)}}{\abs{\sfP_{ij}^\waterh} +\epsilon
    \sfH_{i,\max}^n},1\Big),  & \text{if }\sfH_{i,\max}^n<\sfH(\bsfU_i\upLnp)
    + \sfP_{ij}^\waterh.
  \end{cases}
\end{equation}
This process guarantees that $\Psi_{1}(\bsfU_i\upLnp + \ell\Pij)\geq0$ and
$\Psi_{2}(\bsfU_i\upLnp + \ell\Pij)\geq0$ for all
$\ell\in[0,\ell_j^{i,\waterh}]$.  This enforces a local minimum principle
and a local maximum principle on the water depth. As a corollary this also
enforces positivity of the water depth $\sfH^{n+1}_i$.

After limiting the water depth, we limit the velocity based on the
bound~\eqref{K_bound:Lem:local_bounds}. Notice that the functional
$\Psi_3(\bsfU_i\upLnp + \ell \bsfP^n_{ij})$ is quadratic in $\ell$:
\[\Psi_3(\bsfU_i\upLnp + \ell \bsfP^n_{ij}) = (\sfH_i\upLnp + \ell
  \sfP^h_{ij})^2\sfV_{i,\max}^{2,n} - \|\bsfQ_i\upLnp + \ell
  \bsfP^{\bq}_{ij}\|^2_{\ell^2}.\] Thus, one can find the root
  $\ell_j^{i,\bv}\in[0,\ell_j^{i,\waterh}]$ of
  $\Psi_3(\bsfU_i\upLnp + \ell_j^{i,\bv} \bsfP^n_{ij}) = 0$ by either
  solving the quadratic equation as in
  \citep[Eq.~(6.33)-(6.34)]{Guermond_2018_SISC} or simply employing a
  quadratic Newton algorithm. We refer the reader
to~\citep[Alg.~3]{ryujin-2021-1} for a description of the quadratic
newton algorithm implemented in the code used for the numerical
illustrations. This process guarantees that
$\Psi_3(\bsfU_i\upLnp + \ell \bsfP_{ij}^n)\geq0$ for all
$\ell\in[0,\ell_j^{i,\bv}]$. This enforces a local maximum principle
on the quantity $\|\bv\|^2_{\ell^2}$. As a corollary, this also
enforces that the final solution will be well-balanced with respect to
rest states.

Finally, we set the optimal limiting coefficient to
\begin{equation}\label{eq:final_limiter}
  \ell_{ij}^n\eqq\min(\ell_j^{i,\bv}, \ell_i^{j,\bv}), \qquad \text{for all
  }i\in\calV\text{ and }j\in\calI(i),
\end{equation}
to ensure conservation of the method.

\subsubsection{Conservation, invariant-domain preservation and well balancing}
We now formalize results for the convex limiting procedure concerning
conservation, invariant-domain preservation and well balancing.
\begin{proposition}[Conservation]\label{prop:conservation}
The update given by~\eqref{eq:final_update} is mass conservative up to the contribution of external sources.  
\end{proposition}
\begin{proof}
  Assume the bathymetry is flat and there is no contribution of external
  source terms. Then, given the facts that the limiter $\ell_{ij}$ is
  symmetric by definition and the quantity $\Pij$ is skew-symmetric yields
  $\sum_{i\in\calV}m_i\bsfU_i\upnp = \sum_{i\in\calV}m_i\bsfU_i\upn$.
\end{proof}
\begin{proposition}[Invariant-domain preserving]\label{prop:IDP}
  Let  $n\geq0$. Assume that the $\bsfU_i\upn\in\calA$ for all $i\in\calV$.
  Then the update $\bsfU_i\upnp$ given by~\eqref{eq:final_update} with the
  limiting coefficient~\eqref{eq:final_limiter} is invariant-domain
  preserving under the time-step restriction $\dt \le
  \min_{i\in\calV}\frac{m_i}{2|d_{ii}\upLn|}$.
\end{proposition}
\begin{proof}
  Suppose that the time-step restriction $\dt \le
  \min_{i\in\calV}\frac{m_i}{2|d_{ii}\upLn|}$ holds. Then, the 
  combination in Lemma~\ref{Lem:low_order_stability} is convex and the
  local bounds in Lemma~\ref{Lem:local_bounds} hold true. Then, by
  construction of the limiter~\eqref{eq:lij_h}, we have that:
  \begin{align*}
    \sfH\big(\bsfU_i\upnp\big) = \sfH\Big(\ssum\lambda_i\big(\bsfU_i\upLnp +
    \ell_{ij}^n \Pij^n\big)\Big)\geq\sfH_{i,\min}^n>0.
  \end{align*}
  Thus, $\bsfU_i\upLnp\in\calA$ for all $i\in\calV$.
\end{proof}
\begin{proposition}[Well balancing]\label{prop:well-balancing}
  Let $n\geq0$ and assume that the given state
  $\{\bsfU\upn_i\}_{i\in\calV}\subset\calA$ is at rest as
  formalized in Def.~\ref{Def:Rest_at_large}. Then, the update
  $\bsfU_i\upnp$ given by~\eqref{eq:final_update} with the limiting
  coefficient~\eqref{eq:final_limiter} is at rest under the
  time-step restriction
  $\dt \le \min_{i\in\calV}\frac{m_i}{2|d_{ii}\upLn|}$.  (This
    means that the scheme is well-balanced with respect to rest
    states.)
\end{proposition}
\begin{proof}
  Assume that at $t^n$ the discrete state, $\{\bsfU_i\upn\}_{i\in\calV}$,
  is at rest in the sense of Def.~\ref{Def:Rest_at_large}.
  By Lemma~\ref{lem:low-order-WB}, the low-order update $\bsfU\upLnp$ is
  at rest. By assumption, we have that $\bsfV_i\upn=\bm{0}$ for
  all $i\in\calV$ and so $\sfV_{i,\max}^{2,n} = 0$. Then, the limiting
  strategy for $\Psi_3(\bsfU_i\upLnp + \ell \bsfP_{ij}^n) = 0$ reduces
  to finding $\ell$ such that
  $\|\bsfQ_i\upLnp + \ell \bsfP^{\bq}_{ij}\|^2_{\ell^2} = 0$.
   As $\bsfQ_i\upLnp=\bzero$, we infer that
    $\ell^2\|\bsfP^{\bq}_{ij}\|^2_{\ell^2} = 0$.  If
    $\|\bsfP^{\bq}_{ij}\|^2_{\ell^2} = 0$, every value $\ell\in [0,1]$
    gives $\|\bsfQ_i\upLnp + \ell \bsfP^{\bq}_{ij}\|^2_{\ell^2} = 0$,
    otherwise one must have $\ell=0$ and the same conclusion holds.
  Thus, the final update~\eqref{eq:final_update} reduces to the
  low-order solution $\bsfU\upLnp$ which is well-balanced.
\end{proof}


\section{High-order IDP time-stepping: algorithmic and implementation details}
\label{sec:erk_IDP}

In order to achieve a higher order approximation in time the simple
convex-limited Euler update \eqref{eq:final_update} is now used as a
building block for a higher order explicit Runge Kutta scheme. To
ensure robustness of the method it is crucial that the high-order
Runge-Kutta update is also invariant domain preserving. A widely used
family of Runge-Kutta schemes achieving this are \emph{strong stability
preserving} (SSP) explicit Runge-Kutta (ERK) methods introduced by
\cite{Shu_Osher1988}; see also \citep{Gottlieb_Shu_Tadmor_2001,
Shu_Xing_2014}. Here,
we choose a slightly different approach by using a family of
\emph{invariant-domain preserving} (IDP) explicit Runge-Kutta
methods~\citep{Ern_Guermond_2022} that have the distinct advantage of
having a milder time step size restriction than SSP-ERK methods. We refer
the reader to \citep{Ern_Guermond_2022} for a detailed discussion about
derivation and design of IDP-ERK methods. An idea of ensuring stability
through limiting of ERK stages has also been proposed
in~\cite[Sec.~3.3]{kuzmin2022bound}.

For the sake of completeness, we first recall the general setup and formulas
for ERK methods. Let $s\geq1$ be the number of stages. Let
$\frac{\dif}{\dif t}\bu = \calL(t,\bu)$ denote a generic ordinary
differential equation. Then, the s-stage ERK method for solving the ODE is
given by: $\bu^{n, l}\eqq\bu^n + \dt_{\text{ERK}}
\sum_{j\in\{1:l-1\}}a_{l,j}\,\calL(t^n + c_j\dt_{\text{ERK}}, \bu^{n,j})$
for all $l\in[1:s]$ and $\bu^{n+1}\eqq\bu^n + \dt_{\text{ERK}}
\sum_{j\in\{1:s\}}b_{j}\,\calL(t^n + c_j\dt_{\text{ERK}}, \bu^{n,j})$. The
coefficients of the method are typically recorded in a Butcher tableau:
\begin{center}
  \begin{tabular}{ c|c c c c c}
      $c_1$    & 0         &                                            \\
      $c_2$    & $a_{2,1}$ & 0                                          \\
      $c_3$    & $a_{3,1}$ & $a_{3,2}$ & 0                              \\
      $\vdots$ & $\vdots$  &           & $\ddots$ & $\ddots$            \\
      $c_s$    & $a_{s,1}$ & $a_{s,2}$ & $\cdots$ & $a_{s,s-1}$ & 0     \\
      \hline
               & $b_1$     & $b_2$     & $\cdots$ & $b_{s-1}$   & $b_s$
  \end{tabular}
\end{center}
These coefficients satisfy various consistency criteria which we omit here
for brevity. Recall that the coefficients $c_j$ define the intermediate
time steps $t^{n,j}\eqq t^n + c_j \dt_{\text{ERK}}$. In the following we
focus on a family of second to fourth order ERK methods with \emph{optimal
efficiency ratio} \citep{Ern_Guermond_2022}, meaning, $c_1\eqq 0$ and
$c_{l}-c_{l-1} \eqq \tfrac{1}{s}$. Such methods are optimal in the sense
that the step size of the combined $s$-stage ERK update is
$\dt_{\text{ERK}}=s\dt$, where $\dt$ is the corresponding step size of a
single low-order explicit Euler step. We adopt the notation $\text{RK}(s,
p; 1)$ from \citep{Ern_Guermond_2022}, where $s$ is the number of stages,
$p$ the order of accuracy and $c_{\text{eff}}=1$ is the efficiency ratio.

We now present a reformulation of the IDP-ERK paradigm specialized for this
family of optimal efficiency ratio ERK methods, that is particularly
suitable for a high-performance implementation.
Given a state vector $\bsfU^n$ at time $t^n$ and a (single-step) time-step
size $\tau_n$ satisfying the step size restriction of
Lemma~\ref{Lem:low_order_stability}, we construct a sequence of
updates as follows:
\begin{align}
  \xymatrix{
    \bsfU^n \qqe \bsfU^{(1)} \ar@/_1pc/[r]_{+\,\tau_n} &
    \bsfU^{(2)} \ar@/_1pc/[r]_{+\,\tau_n} &
    \phantom{\bsfU^{(l)}}\ldots\phantom{\bsfU^{(l)}} \ar@/_1pc/[r]_{+\,\tau_n} &
    \bsfU^{(s)} \ar@/_1pc/[r]_{+\,\tau_n} &
    \bsfU^{(s+1)} \qqe \bsfU^{n+1}.}
\end{align}
Notice that here we use the elementary time step $\dt_n$ instead of the
global time step $\dt_{\text{ERK}}$. Recall that $\dt_{\text{ERK}}=s\dt_n$.
Let us introduce the notation $a_{kk}\eqq0$ and $a_{s+1,k}\eqq b_k$ for all
$k\in\intset{1}{s}$. Then we define the weights $w^{(l)}_{k}$ for
$l\in\intset{1}{s}$, $k\in\intset{1}{l}$ as follows:
\begin{align}
  w^{(l)}_{k} \eqq s\,(a_{l+1,k}-a_{l,k}).
\end{align}
We start with $\bsfU^{(1)}\eqq \bsfU^n$. Then, for $l\in[1:s]$, we compute
the $(l+1)$-th stage vector $\bsfU^{(l+1)}$ at time $t^n+l\tau_n$
with the following procedure:
\begin{itemize}
  \item
    Using the previous stage $\bsfU^{(l)}$, we compute the low-order fluxes
    $\bF^{\textup{L},(l)}_{ij}$ with equation~\eqref{eq:low_order_flux}.
    Then we compute the low-order update $\bsfU^{\textup{L},(l+1)}$
    with~\eqref{eq:low_order_scheme} and the time step size $\tau_n$.
  \item
    Using the previous stage $\bsfU^{(l)}$ again, we compute the
    high-order fluxes $\bF^{\textup{H},(l)}_{ij}$ with
    equations~\eqref{eq:high_order_flux}, store these fluxes with the
    previous ones, $\bF_{ij}^{\textup{H},(1)}$, $\ldots,$
    $\bF_{ij}^{\textup{H},(l-1)}$, and set
    \begin{align*}
      \breve\bsfF_{ij}&\eqq
      \sum_{k=1}^{l} w^{(l)}_{k}\FHk_{ij} \quad\text{and}\quad
      \breve\bsfF_{i}=\ssum\breve\bsfF_{ij}.
    \end{align*}
  \item
    Next we compute the fluxes $\breve\bsfP_{ij}$ as in~\eqref{eq:P_ij}:
    \begin{align*}
      \breve\bsfP_{ij} &\eqq \frac{\tau }{m_i
      \lambda_i}\big\{\breve\bsfF_{ij} -
     \FLl_{ij} + b_{ij}\breve\bsfF_j -
      b_{ji}\breve\bsfF_i\big\}.
    \end{align*}
  \item
    Finally, we compute the limiter coefficients $\breve\ell_{ij}$ as
    outlined in Section~\ref{sec:convex_limiting} using
    $\breve\bsfP_{ij}$ into~\eqref{eq:final_update}, and we define
    the high-order update $\bsfU^{(l+1)}$ by setting
    \begin{align*}
      \bsfU_i^{(l+1)} = \bsfU_i^{\textup{L},(l+1)}
      + \snsum\breve\ell_{ij}\lambda_i\breve\bsfP_{ij}.
    \end{align*}
\end{itemize}
The procedure described above inherits at every stage  the properties
listed in Propositions~\ref{prop:conservation},~\ref{prop:IDP},
and~\ref{prop:well-balancing}.


\section{Numerical illustrations}
\label{sec:illustrations}
In this section, we illustrate the proposed method with various
configurations including: \textup{(i)} well-balancing tests; \textup{(ii)}
validation tests for convergence; \textup{(iii)} verification with
small-scale laboratory experiments; \textup{(iv)} realistic flooding scenario with a digital elevation model.

\subsection{Technical details}
The numerical tests are conducted using
the high-performance finite element
code,~\texttt{ryujin}~\citep{ryujin-2021-1, ryujin-2021-3}. The code
uses continuous $\polQ_1$ finite elements on quadrangular meshes for the
spatial approximation and is built upon the \texttt{deal.II} finite element
library~\citep{dealII95}.

To differentiate the temporal approximations, we use the
notation~$\text{RK}(s, p;c_{\text{eff}})$.  The efficiency ratio for the
IDP-ERK schemes introduced in Section~\ref{sec:erk_IDP} is
$c_{\text{eff}}=1$. All the methods with optimal efficiency used in the paper
are summarized in Table~\ref{tab:erks}.
We are also going to use the standard SSP-ERK method
denoted by $\text{RK}(2, 2; \tfrac12)$ and $\text{RK}(3, 3; \tfrac13)$
(see:~\citep[Eq.~2.16]{Shu_Osher1988} and~\citep[Eq.~2.18]{Shu_Osher1988},
respectively). 
\begin{table}[t!]\small \addtolength{\tabcolsep}{-3.2pt}
  \begin{center}
    \subfloat[$\text{RK}(2,2;1)$]{\parbox{0.3\textwidth}{
    \begin{align*}
      \begin{array}{l|rr}
        w^{(l)}_{k} & 0 & 1\\
        \hline\\[-1.0em]
        1 &  1     \\
        2 & -1 & 2
      \end{array}
    \end{align*}}}
    \subfloat[$\text{RK}(3,3;1)$]{\parbox{0.3\textwidth}{
    \begin{align*}
      \begin{array}{l|rrr}
        w^{(l)}_{k} & 0 & 1 & 2\\
        \hline\\[-1.0em]
        1 &            1                     \\
        2 &           -1 &  2                \\
        3 & \sfrac{3}{4} & -2 & \sfrac{9}{4}
      \end{array}
    \end{align*}}}
    \hspace{1em}
    \subfloat[$\text{RK}(4,3;1)$]{\parbox{0.3\textwidth}{
    \begin{align*}
      \begin{array}{l|rrrr}
        w^{(l)}_{k} & 0 & 1 & 2 & 3\\
        \hline\\[-1.0em]
        1 &  1                                                \\
        2 & -1 &            2                                 \\
        3 &  0 &           -1 &              2                \\
        4 &  0 & \sfrac{5}{3} & -\sfrac{10}{3} & \sfrac{8}{3}
      \end{array}
    \end{align*}}}

    \subfloat[$\text{RK}(5,4;1)$]{\parbox{1.0\textwidth}{
    \small 
      \centering \begin{tabular}{l|rrrr}
        $w^{(l)}_{k}$ & 0 & 1 & 2 & 3\\
        \hline\\[-1.0em]
        1 &  1.000000000000000 \\
        2 &  0.303779113477746 &  0.696220886522255 \\
        3 & -2.596605007106260 &  3.860592821791782 & -0.263987814685521 \\
        4 &  2.373989715203703 & -1.980102553333916 & -3.819151895277756 &  4.425264733407969 \\
        5 & -1.606747744309784 &  1.817291202624922 &  1.137969506889054 & -2.114595709136266 \\
      \end{tabular}
      \\[0.5em]
      \hfill $w^{(5)}_{4}$=1.766082743932075
    }}
  \end{center}
  \caption{Weights $w^{(l)}_{k}$ for different optimal IDP-ERK
    schemes ranging from an (a) two-stage, (b) three-stage, (c) four-stage,
    to (d) a five-stage method.}
  \label{tab:erks}
\end{table}

The time step size $\tau_n$ is computed during the first stage of each time
step using the expression
\begin{align}
  \dt_n \eqq \textup{CFL} \max_{i\in\calV}\frac {m_i}{2|d_{ii}\upLn|},
\end{align}
where $\textup{CFL}\in(0,1]$ is a user-defined constant henceforth called
Courant-Friedrichs-Lewy number. The global time step is computed
using $\dt_{\text{ERK}} \eqq c_{\text{eff}} s \dt_n$.

In all the simulations reported below, we take $g=\mynum{9.81}{m\,s^{-2}}$.
To characterize the convergence properties of the method,
we use the following consolidated error indicator for our tests:
\begin{align*}
    \delta^q(T)\eqq \frac{\|\sfH -
    \waterh_{\text{exact}}(T)\|_{L^q(D)}}{\|\waterh_{\text{exact}}(T)\|_{L^q(D)}}
    + \frac{\|\bsfQ -
    \bq_{\text{exact}}(T)\|_{L^q(D)}}{\|\bq_{\text{exact}}(T)\|_{L^q(D)}},
\end{align*}
where $q\in\{1,\infty\}$.

For the sake of brevity, we omit discussing the performance of the
non-reflecting boundary conditions described in
Appendix~\ref{sec:boundary_conditions}. Overall, the non-reflecting
boundary conditions work well and as expected; no issues were observed
regarding significant feedback or violation of the invariant-domain.

\subsection{Well-balancing tests}

In this section, we verify the well-balancing properties of the numerical
method.

\subsubsection{At rest}
To verify the well-balancing at rest, we adopt the three conical bump
topography configuration introduced in~\cite{Kawahara_1986} and initialize
the water depth to $\sfH_0(\bx) = \max(\mynum{1.5}{m} - z(\bx), 0)$ so that
part of the topography is submerged and some is exposed creating a
shoreline. The computational domain is set to $D = [0,
\mynum{75}{m}]\times[0, \mynum{30}{m}]$ with slip boundary conditions. To
make the problem slightly more challenging, we apply some distortion to the
mesh since most realistic topographical data and respective meshes might
not be uniform. We run until final time $T=\mynum{100}{s}$ with CFL 0.9
using $\text{RK}(3, 3; 1)$ and $\text{RK}(3, 3; \tfrac13)$. As shown
in Figure~\ref{fig:wb-distorted-mesh}, no special
treatment is done to align the shoreline with
the mesh throughout the domain. We report the $L^\infty(T)$-norm of the error on the
water depth for two meshes in Table~\ref{tab:well-balancing}.
Inspection of the table shows that the method is indeed
  well-balanced even when the shoreline does not coincide with the
  mesh,
  which is a key improvement over the method proposed
  in \citep[\S5\&6]{Guermond_2018_SISC}.
\begin{figure}[t!]
  \centering
  \includegraphics[trim={0, 0, 0, 0},clip,width=0.90\textwidth]{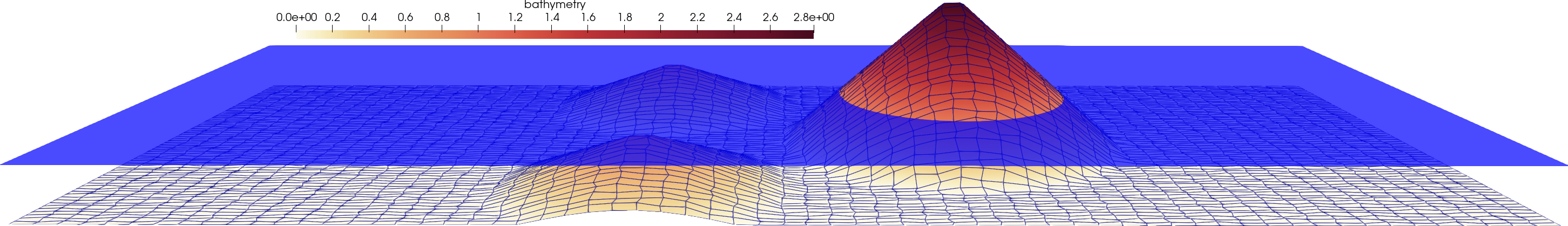}
  \caption{Well-balancing configuration with distorted mesh with 4225
    $\polQ_1$ degrees of freedom.}\label{fig:wb-distorted-mesh}
\end{figure}
\begin{table}[t!]
  \begin{minipage}[t]{0.49\textwidth}
    \centering
    \begin{tabular}{c c c}
      \toprule
      $I$   & $\text{RK}(3, 3; 1)$ & $\text{RK}(3, 3; \tfrac13)$ \\[0.25em]
      4225  & \SI{1.33e-15}{}      & \SI{9.44e-14}{}                   \\
      16641 & \SI{1.33e-15}{}      & \SI{2.12e-13}{}                   \\
      \bottomrule
    \end{tabular}
    \caption{At-rest well-balancing results.}
    \label{tab:well-balancing}
  \end{minipage}
  \begin{minipage}[t]{0.49\textwidth}
    \centering
    \begin{tabular}{c c c}
      \toprule
      $I$  & $\text{RK}(3, 3; 1)$         & $\text{RK}(3, 3; \tfrac13)$ \\[0.25em]
      513  & \SI{6.61675179979799e-14}{}  & \SI{5.299568154477178e-13}{}      \\
      1025 & \SI{1.642182047833548e-14}{} & \SI{6.25744414405246e-13}{}       \\
      \bottomrule
    \end{tabular}
    \caption{Steady flow over inclined plane well-balancing
      results.}
    \label{tab:well-balancing-incline}
  \end{minipage}
\end{table}

\subsubsection{Steady flow over inclined plane with friction}
We now test the well-balancing property for a steady flow over an inclined
plane with Gauckler-Manning friction. The specific configuration that we
consider is that proposed in~\cite[Sec.~4.1]{Chertock_Kurganov_2015} titled
``Example 1'' (Test 2). The domain is set to $D = (0,\mynum{25}{m})$ with
Dirichlet conditions on the left for inflow, and non-reflecting boundary
conditions on the right for dynamic outflow. The topography profile
is defined by $z(x) = -b x$. The unit discharge is initialized with $q(x) =
q_0$. The initial and exact solution for the depth is given by
$\waterh(x)\equiv \waterh_0 = (\frac{n^2q_0^2}{b})^{3/10}$ where $n$ is the
Gauckler-Manning friction coefficient. ; see (2.5)
in~\citep{Chertock_Kurganov_2015}, The coefficients are set to $b =
\mynum{0.01}{}, q_0 = \mynum{0.1}{m^2/s}, n = \mynum{0.02}{m^{-1/3}s}$
which gives approximately $\waterh_0\approx\mynum{0.0956352}{m}$. We run
until final time $T=\mynum{100}{s}$ with CFL 0.5 using $\text{RK}(3, 3; 1)$
and $\text{RK}(3, 3; \tfrac13)$. We report the $\delta^\infty(T)$
error for two meshes in Table~\ref{tab:well-balancing-incline}.
Well-balancing is again achieved in this case.

\subsubsection{Rainfall over inclined plane with friction}
\begin{figure}[t]
  \centering
  \includegraphics[trim={0, 0, 5, 0},clip,width=0.49\textwidth]{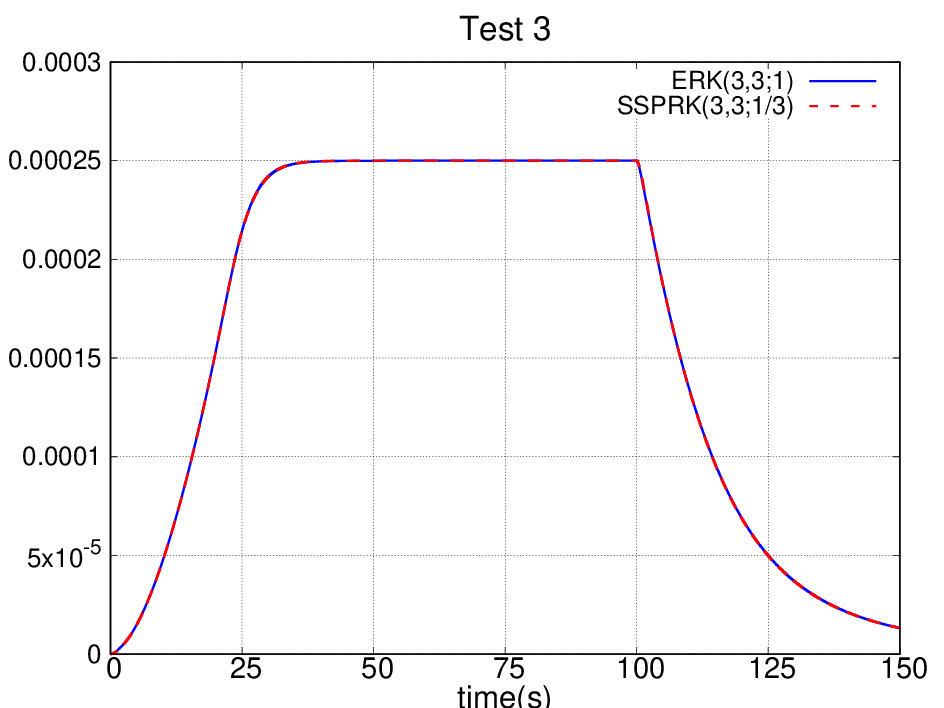}
  \includegraphics[trim={0, 0, 5, 0},clip,width=0.49\textwidth]{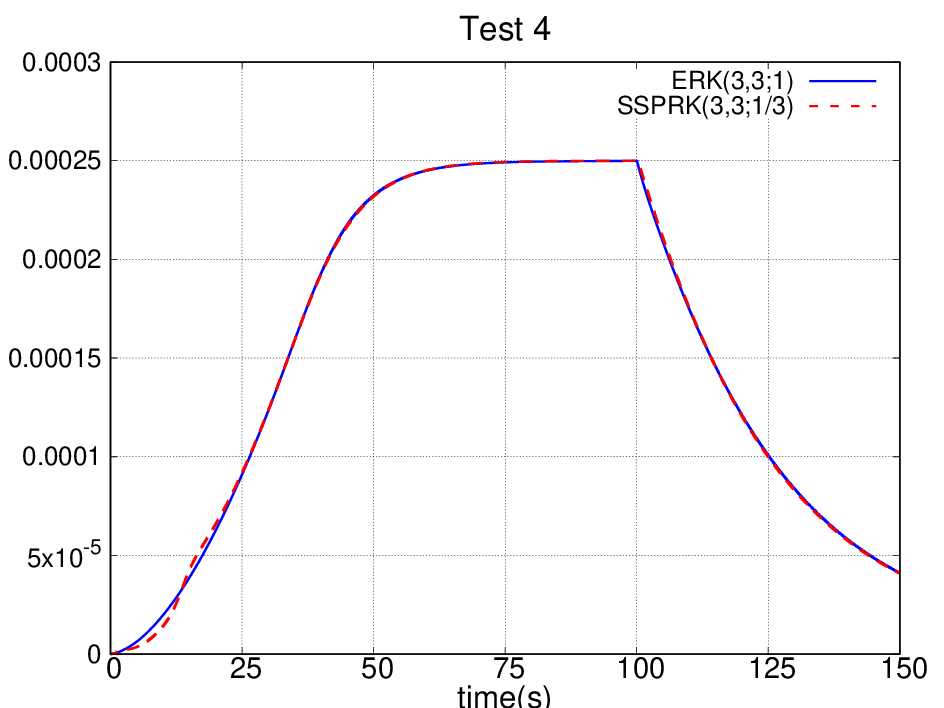}
  \caption{Comparison of discharge at outlet boundary ($x =
    \mynum{2.5}{m}$) for the ``Test 3'' and ``Test 4'' configurations from
    Table 2 in~\citep{Chertock_Kurganov_2015}.}
  \label{fig:rain-test}
\end{figure}
We now test the method's ability to handle both rainfall and friction
effects as sources. We again use the inclined plane bathymetry from the
previous section, but now follow the configuration
in~\citep[Sec.~4.1]{Chertock_Kurganov_2015} titled ``Example 3''. Here, the
initial configuration is set to dry with a constant rain source $R(x) =
\mynum{1e-4}{m s^{-1}}$ active in the interval $0\leq t\leq\mynum{100}{s}$.
The specific test cases we reproduce are ``Test 3'' and ``Test 4'' from
Table 2 in~\citep{Chertock_Kurganov_2015}. The domain is set to $D =
(0,\mynum{2.5}{m})$ with slip conditions on the left and do nothing
boundary conditions on the right. For both tests, we run until final time
$T=\mynum{150}{s}$ with CFL 0.5 using $\text{RK}(3, 3; 1)$ and
$\text{RK}(3, 3; \tfrac13)$.The discharge was measured over time at the right
boundary for the simulation for each test. We report the time history
of the discharge in
Figure~\ref{fig:rain-test}. We observe comparable results to those reported
in~\citep[Fig.~5]{Chertock_Kurganov_2015}. For the $\text{RK}(3, 3;
\tfrac13)$ results in Test 4, there are some slight oscillations at the
beginning of the simulation.

\subsection{Convergence tests}

In this section, we verify the accuracy of the proposed method. For the
sake of brevity, we only report results in two space dimensions since we observe
similar behavior in one space dimension.

\subsubsection{Smooth vortex}
\label{sec:vortex}
\begin{table}[t]
  \centering
  \begin{tabular}{ccccc}
    \toprule
    $I$     & $\text{RK}(3, 3; 1)$ & rate & $\text{RK}(3, 3; \tfrac13)$ & rate \\[0.25em]
    1089    & \num{0.00357866}     & -    & \num{0.00357241}                  & -    \\
    4225    & \num{0.000628104}    & 2.51 & \num{0.000627444}                 & 2.51 \\
    16641   & \num{8.41359e-05}    & 2.90 & \num{8.39931e-05}                 & 2.90 \\
    66049   & \num{1.09506e-05}    & 2.94 & \num{1.09353e-05}                 & 2.94 \\
    263169  & \num{1.42512e-06}    & 2.94 & \num{1.42405e-06}                 & 2.94 \\
    1050625 & \num{2.81145e-07}    & 2.34 & \num{ 1.91891e-07}                & 2.89 \\
    \bottomrule
  \end{tabular}
  \caption{Error $\delta^1(T)$ and convergence rates for smooth vortex test
    with CFL 0.25.}
  \label{tab:vortex-L1}
\end{table}
We now demonstrate the convergence of the method with a smooth analytical
solution of the Shallow Water Equations. This benchmark is a
divergence-free vortex adapted (and slightly modified)
from~\cite[Sec.~2.3]{Ricchiuto_Bollermann_2009} which mimics geophysical
flows~\citep{ricchiuto2007application}. Let $(\waterh_\infty, \bv_\infty)$
be the far-field state. Then, the analytical solution is defined as
follows:
\begin{subequations}\label{unsteady-vortex}
  \begin{align}
    & \waterh(\bx, t) = \waterh_\infty -\frac{1}{2 g r_0^2}\psi(\overline{\bx})^2,                                 \\
    & \bv\eqq\bv_\infty + \delta\bv,                                                                               \\
    & \delta\bv(\bx, t)\eqq\left(\partial_{x_2}\psi(\overline{\bx}),-\partial_{x_1}\psi(\overline{\bx})\right)\tr,
  \end{align}
\end{subequations}
with $\overline{\bx}\eqq\bx - \bx^0 - \bv_\infty t$ and
$\psi(\bx)\eqq\frac{\beta}{2\pi}\exp(\tfrac12(1 -
\frac{\|\bx\|^2_{\ell^2}}{r_0^2}))$. Here, $\bx^0$ can be thought of as the
center of the vortex, $\beta$ the vortex strength and $r_0$ the radius of
the vortex. The parameters are set to $\waterh_\infty =
\mynum{2}{m},\,\beta = \mynum{2},\,r_0 = \mynum{1}{m},\,\bv_\infty =
(1,1)\mynum{}{m s^{-1}}$. The computational domain is set to
$D=(-6,\mynum{6}{m})\times(-6,\mynum{6}{m})$ with Dirichlet boundary
conditions. We set the final time to $T=\mynum{2}{s}$. The time-stepping is performed with $\text{RK}(3, 3; 1)$ and $\text{RK}(3, 3; \tfrac13)$
with CFL 0.25. We report the consolidated $\delta^\infty(T)$ error and rates in
Table~\ref{tab:vortex-L1}. We observe close to third order accuracy in
  time and space. The super-convergence in space is compatible with the
theoretical result from \citep[Prop.~A.1]{Guermond_Pasq_2013}.

\subsubsection{Planar surface flow in paraboloid-shaped basin}
\begin{table}[t]
  \centering
  \begin{tabular}{ccccc}
    \toprule
    $I$     & $\text{RK}(3, 3; 1)$ & rate & $\text{RK}(3, 3; \tfrac13)$ & rate \\[0.25em]
    1089    & \num{0.221698}       & -    & \num{0.22683}                     & -    \\
    4225    & \num{0.0632759}      & 1.81 & \num{0.0647327}                   & 1.81 \\
    16641   & \num{0.0172272}      & 1.88 & \num{0.0178517}                   & 1.86 \\
    66049   & \num{0.00510551}     & 1.75 & \num{0.00535852}                  & 1.74 \\
    263169  & \num{0.0017404}      & 1.55 & \num{0.00183409}                  & 1.55 \\
    1050625 & \num{0.000704076}    & 1.31 & \num{0.000738643}                 & 1.31 \\
    \bottomrule
  \end{tabular}
  \caption{Error $\delta^1(T)$ and convergence rates for the test
    configuration with a planar surface flow in a paraboloid-shaped basin.}
  \label{tab:parab}
\end{table}
We now demonstrate the convergence of the method with Thacker's planar
surface flow in paraboloid-shaped basin~\citep{thacker1981}. The
problem consists of a free-surface moving in a periodic motion inside
a paraboloid-shaped basin.The moving shoreline is circular at all
  times. The precise configuration we use is the one introduced
in~\cite[Sec.~4.2.2]{delestre2013swashes} subsection ``Planar surface
in a paraboloid:'' The computational domain is defined as
$D = [0, \mynum{4}{m}]\times[0,\mynum{4}{m}]$ with slip boundary
conditions. The theoretical period of the motion is
  $2\pi/\sqrt{2 g \waterh_0}$ with $\waterh_0=\mynum{0.1}{m}$. The final time is
  three periods, approximately
$T = \qty[round-mode=places, round-precision=8,
scientific-notation=false]{13.45710440}{s}$. The time-stepping is
performed with $\text{RK}(3, 3; 1)$ and $\text{RK}(3, 3; \tfrac13)$
with CFL 0.5. We report the consolidated $\delta^1(T)$ error and rates
in Table~\ref{tab:parab}. We observe a convergence rate 
  ranging from 1.8 to 1.3, which is consistent with what is reported in
  the literature.

\subsection{Small-scale laboratory experiments}

In this section, we simulate two small-scale laboratory experiments
described in~\cite{martinez_2018}. The goal of the experiments was to
provide validation data for shallow water solvers by studying complex
steady and transient flume experiments. The experiments comprised of
transcritical steady flow and dam-break flow around obstacles and complex
beds. In this paper, we reproduce cases ``G2-S.2'' and ``G3-D.1'' described
in Sections 4.3.1 and 4.4.2 in~\citep{martinez_2018}, respectively. We
refer the reader to~\citep{martinez_2018} for a detailed description of the
experimental configuration. The set up can also be found in the source code
for the \texttt{ryujin} software.

\subsubsection{G2-S.2}
\begin{figure}[t]
  \centering
  \includegraphics[trim={0, 0, 0, 0},clip,width=0.95\textwidth]{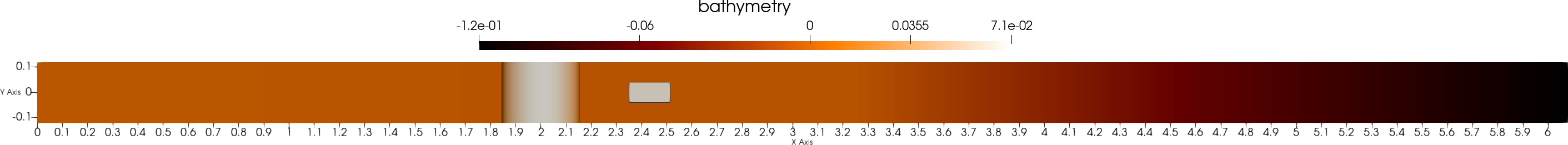}
  \caption{Top view representation for the ``G2-S.2'' test case.}
  \label{fig:G2-setup}

  \centering
  \includegraphics[trim={0, 0, 0, 0},clip,width=0.95\textwidth]{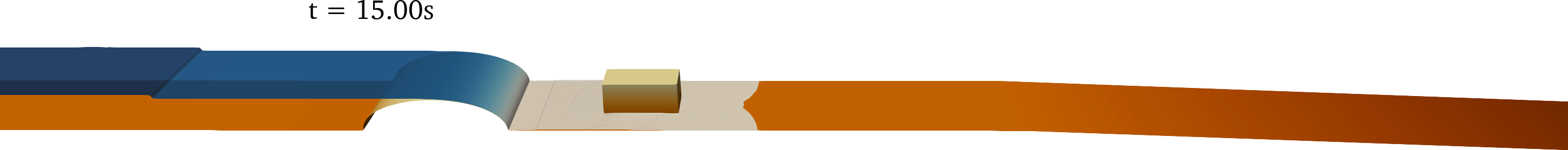}
  \includegraphics[trim={0, 0, 0, 0},clip,width=0.95\textwidth]{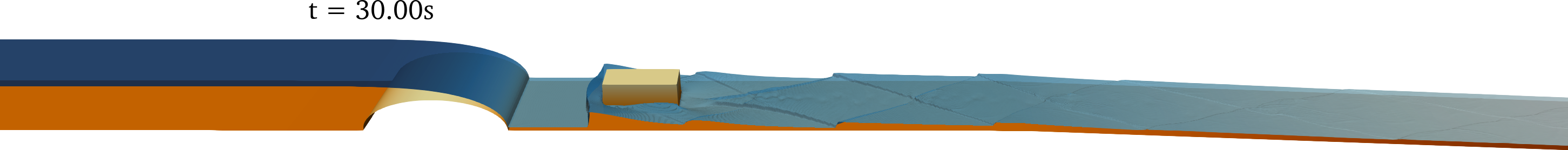}
  \includegraphics[trim={0, 0, 0, 0},clip,width=0.95\textwidth]{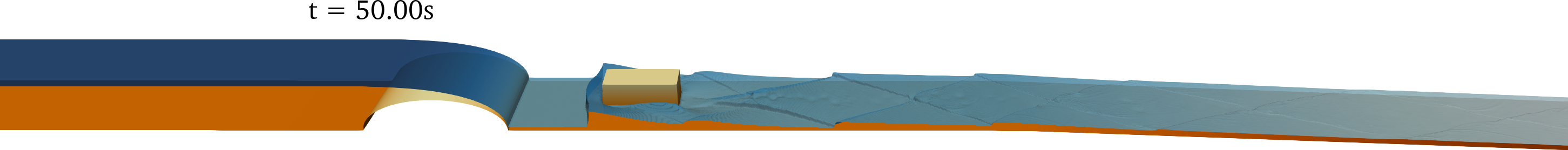}
  \caption{Time snapshots for ``G2-S.2'' showing water elevation and bathymetry.}\label{fig:G2-output}

  \centering
  \includegraphics[trim={0, 10, 25, 20},clip,width=0.48\textwidth]{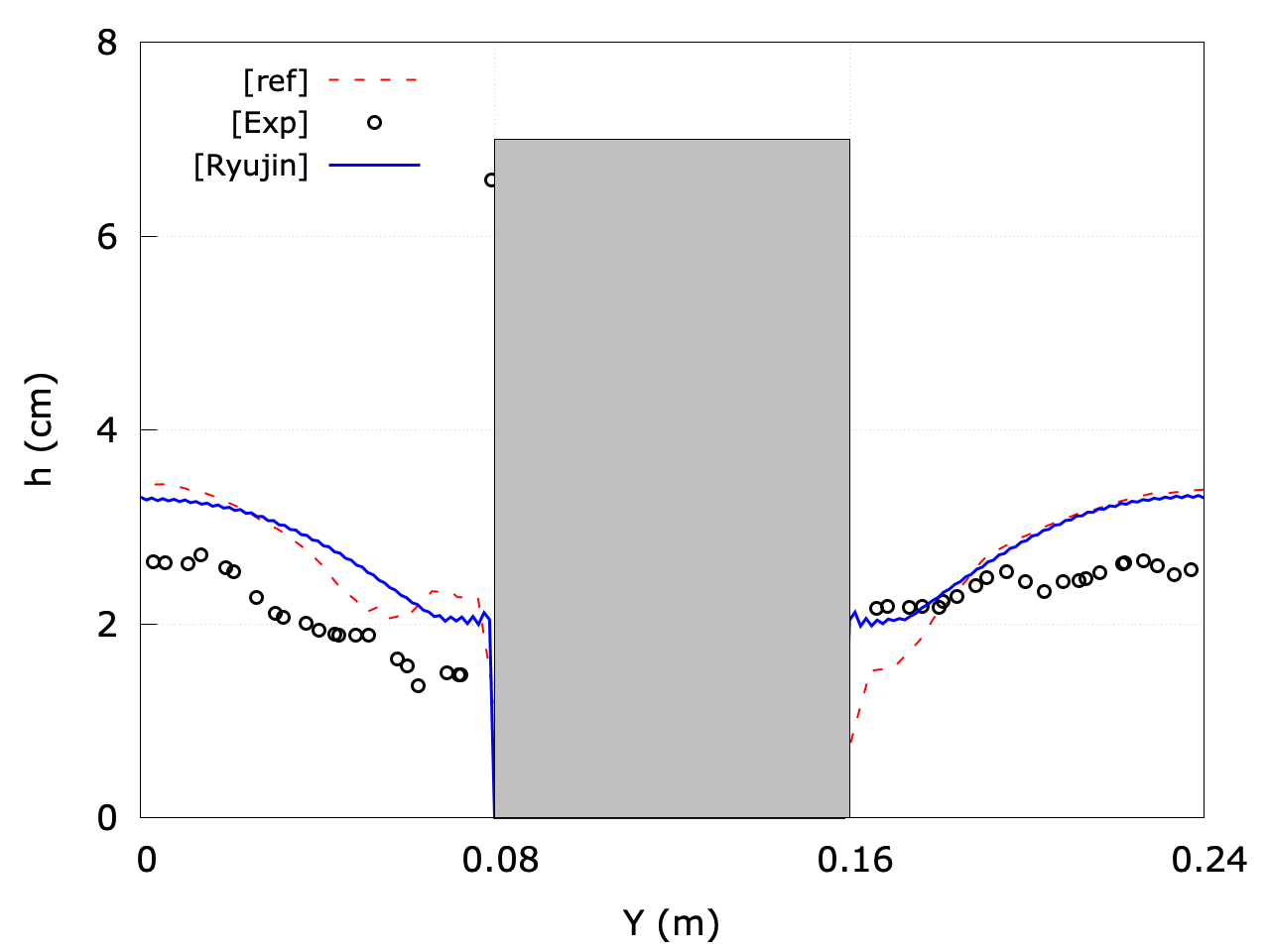}
  \includegraphics[trim={0, 10, 25, 15},clip,width=0.48\textwidth]{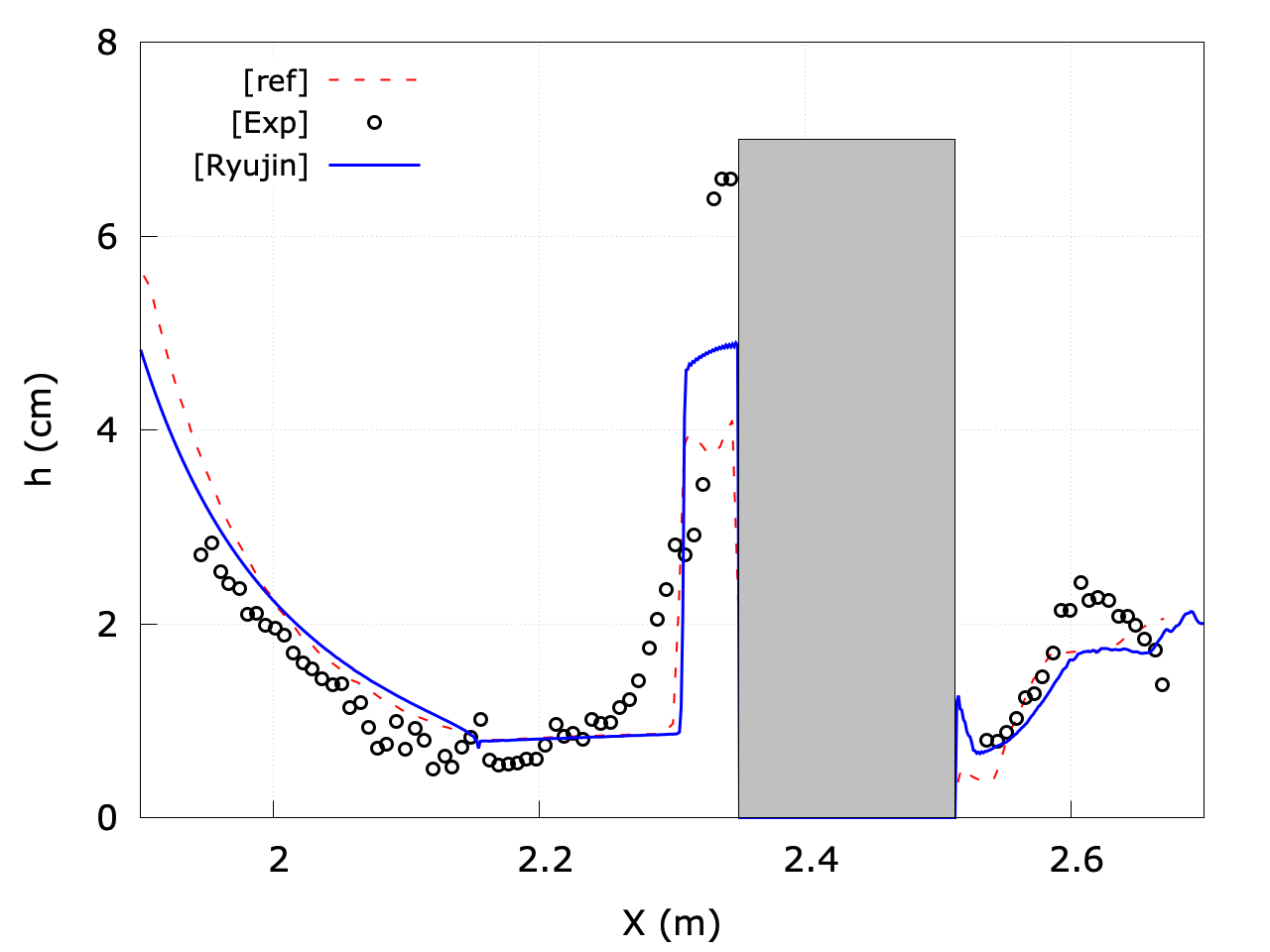}
  \caption{Comparison of numerically computed water depth (solid blue line)
    for the ``G2-S.2'' configuration along the two sections with
    experimental data (black circles) and corresponding simulation data
    from~\citep{martinez_2018} (red dashed line).}
  \label{fig:G2-exps}
\end{figure}
The ``G2-S.2'' test case consists of a steady inflow discharge of $\sfQ_0 =
\mynum{9.01}{m^3 / h}$ with the flume containing a semi-circular bump
across its width followed by a rectangular obstacle placed at the center
line of the flume. We give a top view representation of the set up in
Figure~\ref{fig:G2-setup}.
Note that the discharge here is the volumetric flow rate. For our
simulation, we use the unit flow discharge $q_0 = \sfQ_0 / \mynum{0.24}{m}$
which gives $q_0 = \mynum{0.0104}{m^2 / s}$ (here, the $\mynum{0.24}{m}$
corresponds to the flume width). We reproduce this case using the
computational domain $D = [0, \mynum{6.078}{m}]\times[-0.12,
\mynum{0.12}{m}]$ with 928137 $\polQ_1$ degrees of freedom (this
corresponds to the mesh-size being roughly \mynum{1.25}{mm} in each
direction). Note that we omit discretizing the tank reservoir since it is
not needed for simulating steady inflow. The initial set up consists of a
dry flume where the bottom/top boundaries are set to slip boundary
conditions and the right boundary is set to non-reflecting boundary
conditions for dynamic outflow. On the left boundary, we enforce the steady
inflow discharge and do nothing boundary conditions for the water depth. We
run the simulation with $\text{RK}(4, 3; 1)$ until $T = \mynum{50}{s}$ with
CFL 0.9 to allow the flow to reach a steady state. We output four time
snapshots for $t = \{\SI{15}{s},\SI{30}{s},\SI{50}{s}\}$ in Figure~\ref{fig:G2-output}.

In the experiments, the water depth was measured at two different
sections: $x = \mynum{2.40}{m}$ (across width of flume spanning the
rectangular obstacle) and $y = \mynum{0}{m}$ (centerline of the
flume). In Figure~\ref{fig:G2-exps}, we compare the numerical output
of our simulations along the sections and compare with the
experimental data as well as the simulation data reported
in~\citep{martinez_2018}. Overall, our simulation compares well with
the experiments and simulations
from~\citep{martinez_2018}. The discrepancies between the
  numerical simulations are due to mesh resolution differences. The
  discrepancies with the experiments show the shortcomings of the
  shallow water equations and that, short to solving the Navier-Stokes
  equations, a higher-fidelity model is required.

\subsubsection{G3-D.1}
\begin{figure}[t]
  \centering
  \includegraphics[trim={0, 0, 0, 0},clip,width=0.85\textwidth]{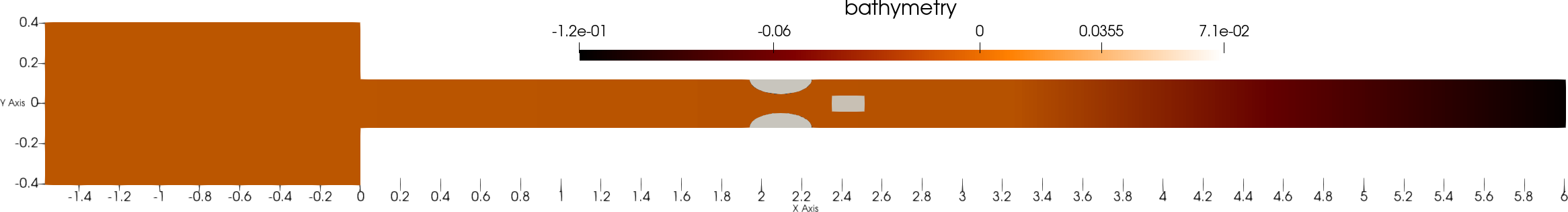}
  \caption{Top view representation for ``G3-D.1'' case.}
  \label{fig:G3-setup}

  \centering
  \includegraphics[trim={0, 0, 0, 0},clip,width=0.49\textwidth]{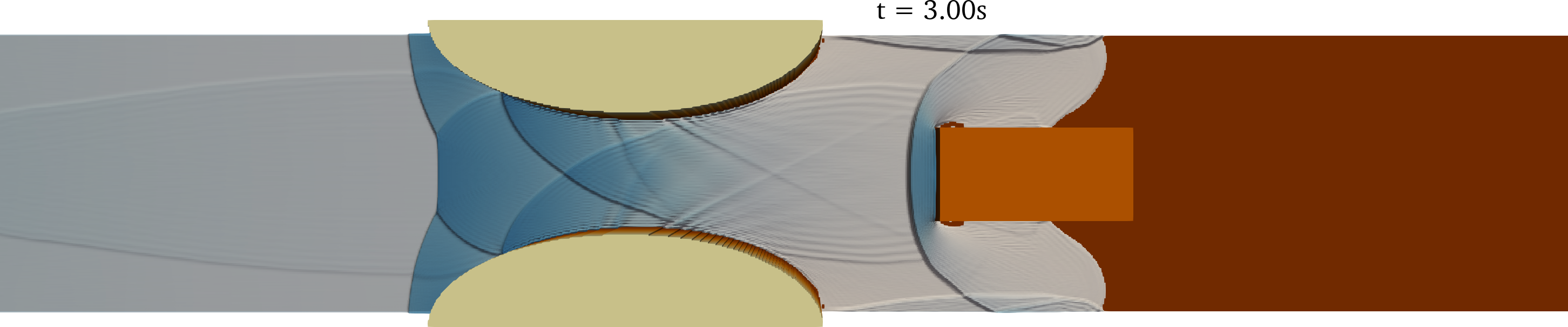}
  \includegraphics[trim={0, 0, 0, 0},clip,width=0.49\textwidth]{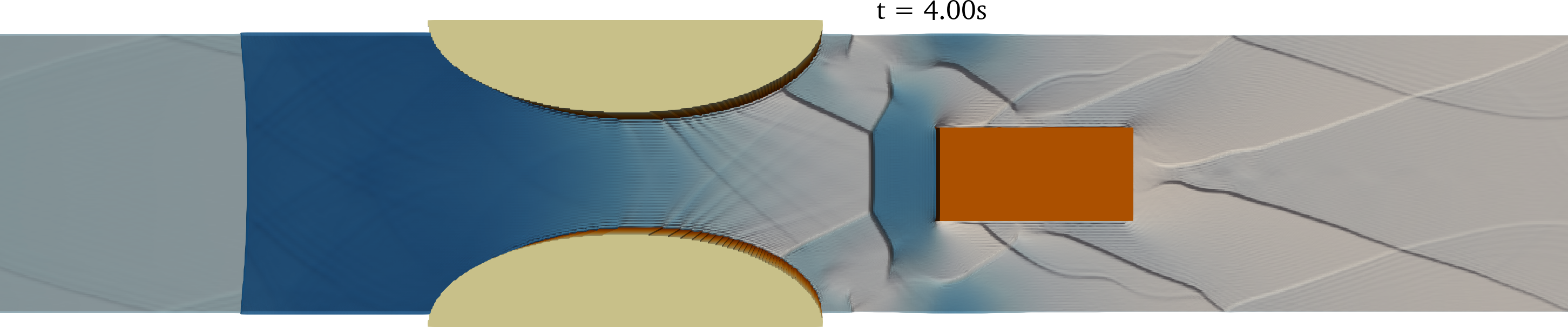}\par
  \includegraphics[trim={0, 0, 0, 0},clip,width=0.49\textwidth]{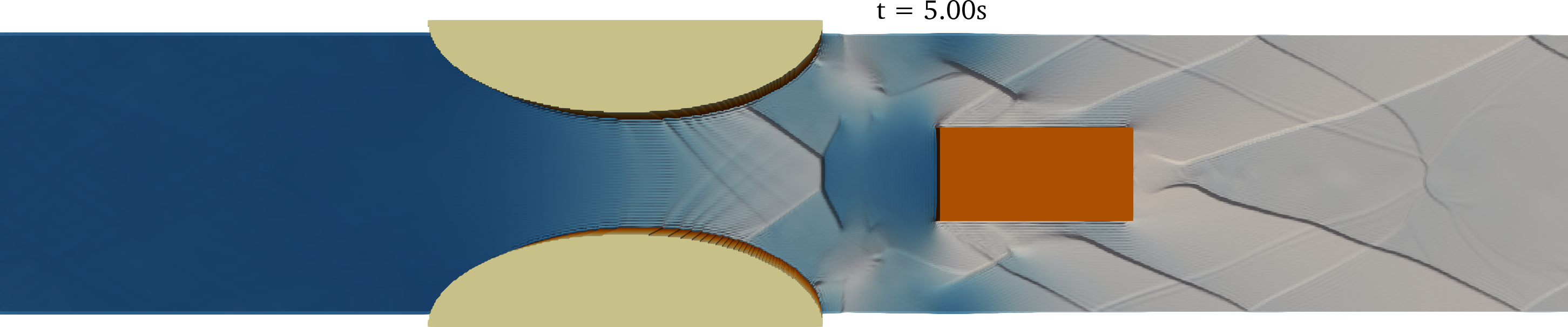}
  \includegraphics[trim={0, 0, 0, 0},clip,width=0.49\textwidth]{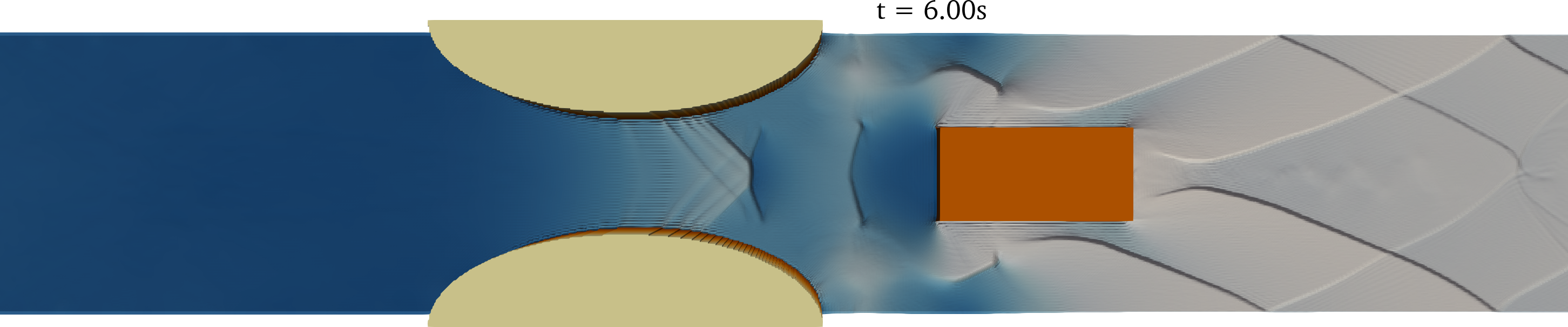}
  \caption{Time snapshots for ``G3-D.1'' showing water elevation and bathymetry.}\label{fig:G3-output}
\end{figure}
The case ``G3-D.1'' consists of a dam-break flow with height
$\sfH_0 = \mynum{0.055}{m}$ in the reservoir and the flume containing
two semi-circular Venturi constriction elements followed by
rectangular obstacle placed at the center line of the flume. We give a
top view representation of the set up in Figure~\ref{fig:G3-setup}. We
reproduce this case using the computation domain
$D = D_\text{res}\cup D_\text{flume}$ where
$D_\text{res} = [-1.58, \mynum{0}{m}]\times[-0.405, \mynum{0.405}{m}]$
and
$D_\text{flume} = [0, \mynum{6.078}{m}]\times[-0.12, \mynum{0.12}{m}]$
with 1753793 $\polQ_1$ degrees of freedom (this corresponds to the
mesh-size being roughly \mynum{1.25}{mm} in each direction). The flume
is initially dry.  The slip boundary condition is enforced on all
  the boundaries except the right-most one.  Non-reflecting boundary
conditions are enforced on the right boundary for dynamic outflow. We
run the simulation with $\text{RK}(4, 3; 1)$ until $T = \mynum{20}{s}$
with CFL 0.9. We output four time snapshots of the water elevation for
$t = \{\SI{3}{s},\SI{4}{s},\SI{5}{s},\SI{6}{s}\}$ in
Figure~\ref{fig:G3-output}.

\begin{figure}[t!]
  \centering
  \includegraphics[trim={0, 10, 0, 20},clip,width=0.70\textwidth]{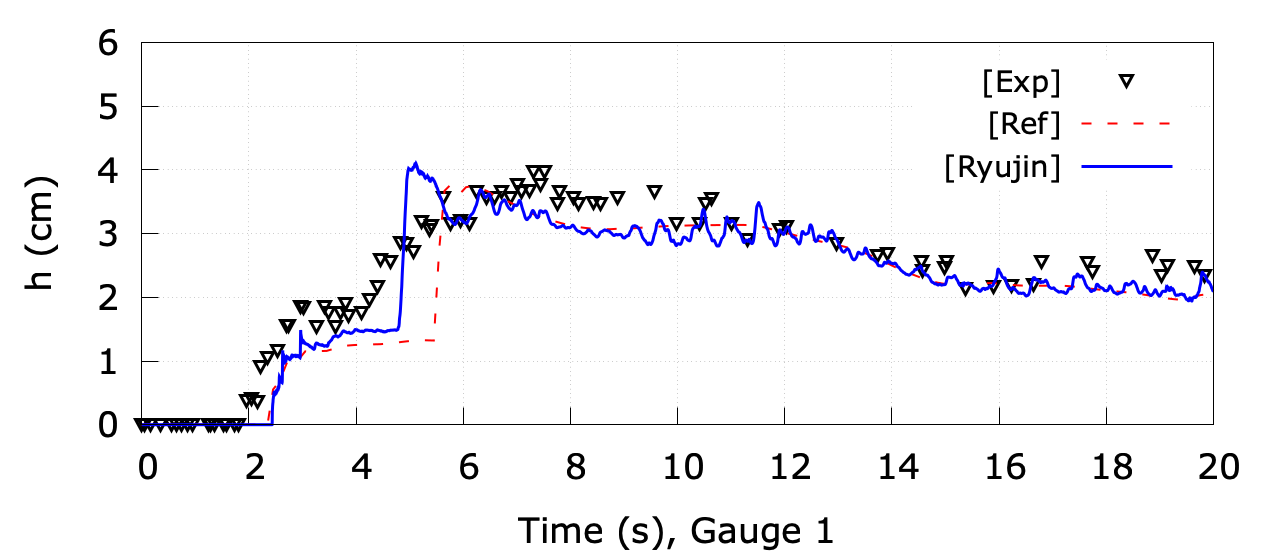}
  \includegraphics[trim={0, 10, 0, 15},clip,width=0.70\textwidth]{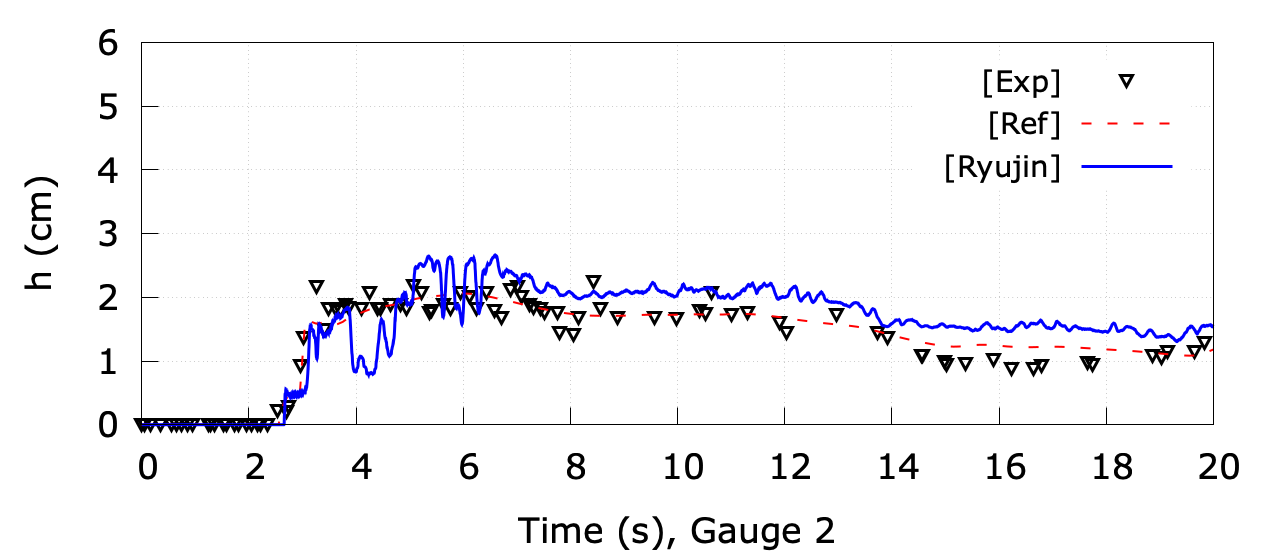}
  \includegraphics[trim={0, 10, 0, 15},clip,width=0.70\textwidth]{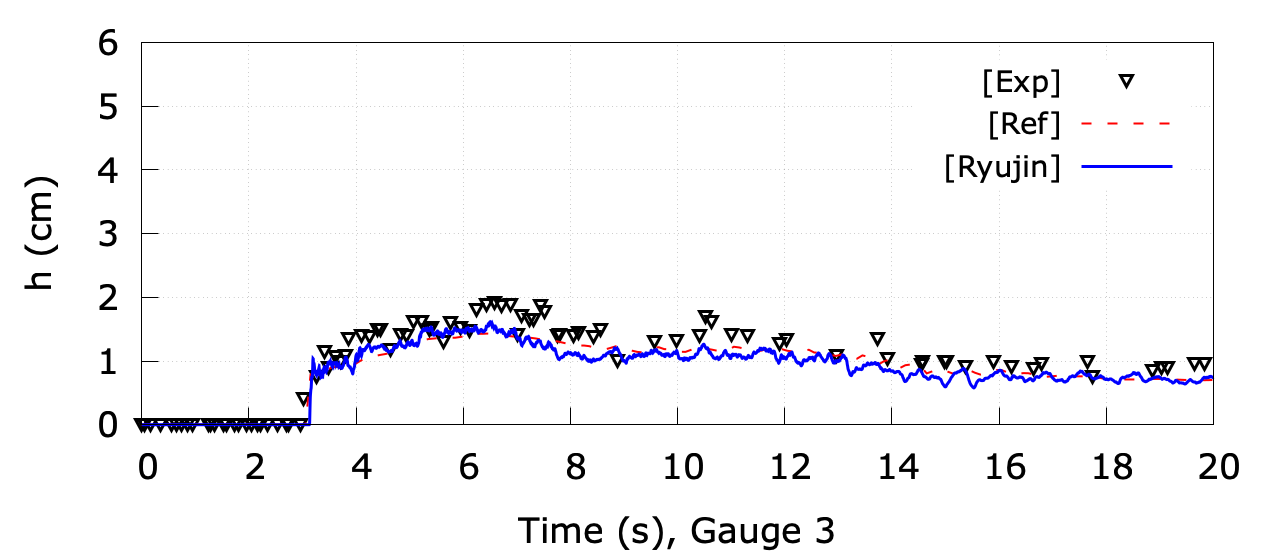}
  \caption{Temporal series over $t\in[0,\mynum{20}{s}]$ comparing numerical
    water depth (blue solid line), experimental data (black triangles) and
    simulation data from~\citep{martinez_2018} (red dashed line). From top
    to bottom: Gauge 1, Gauge 2, Gauge 3.}
  \label{fig:G3-exps}
  \vspace{1.5cm}

  %
  \centering
  \begin{tabular}{lrrr}
    \toprule
    Method                            & \# cycles & throughput            & runtime \\[0.3em]
    $\text{RK}(2, 2; \tfrac12)$ & 1019      & \mynum{1.3564}{ MQ/s} & \mynum{99.64}{s} \\
    $\text{RK}(2, 2; 1)$              & 510       & \mynum{1.2762}{ MQ/s} & \mynum{52.99}{s} \\
    $\text{RK}(3, 3; \tfrac13)$ & 1019      & \mynum{1.3578}{ MQ/s} & \mynum{149.20}{s}\\
    $\text{RK}(3, 3; 1)$              & 340       & \mynum{1.2511}{ MQ/s} & \mynum{54.01}{s} \\
    $\text{RK}(4, 3; 1)$              & 255       & \mynum{1.2312}{ MQ/s} & \mynum{54.91}{s} \\
    $\text{RK}(5, 4; 1)$              & 204       & \mynum{1.1494}{ MQ/s} & \mynum{58.79}{s} \\
    \bottomrule
  \end{tabular}
  \captionof{table}{%
    Efficiency comparison of various time-stepping schemes. We report the
    total number of cycles for a full Runge-Kutta step to reach final
    time $T=\mynum{0.5}{s}$ with CFL 0.2, the average throughput measured
    in \emph{million $Q_1$-mesh points per second} (MQ/s) for a single
    Runge-Kutta substep (consisting of a single forward Euler
    step) and the total runtime to reach the final simulation time.}
  \label{fig:efficiency_tests}
\end{figure}
In the experiments reported in~\citep{martinez_2018}, three wave gauges
were placed in the basin to measure the water depth over the duration of
the experiment. The specific locations of the gauges are given by: ``Gauge
1'' $(\mynum{225}{cm}, \mynum{12}{cm})$; ``Gauge 2'' $(\mynum{240}{cm},
\mynum{20}{cm})$; ``Gauge 3'' $(\mynum{260}{cm}, \mynum{12}{cm})$. In
Figure~\ref{fig:G3-exps}, we compare the numerical output of our
simulations with the experimental data and the
simulation data reported in~\citep{martinez_2018}. Overall, our simulation
compares well with the experiment and simulations from~\citep{martinez_2018}.

\subsection{Efficiency tests}\label{sec:efficiency}
We now report a quick efficiency test to compare various time-stepping
techniques. We choose the smooth vortex benchmark described
in~\ref{sec:vortex} and run until final time $T = \mynum{0.5}{s}$ with CFL
0.2 with the mesh composed of 66,049 $\polQ_1$ degrees of freedom. Each
simulation is performed on a single rank and single thread on a laptop
computer. In Table~\ref{fig:efficiency_tests}, we report the results of our
tests. We see that the overall efficiency
  of each method is directly proportional to the efficiency coefficient $c_{\text{eff}}$.

\subsection{High-fidelity simulation}
\begin{figure}[p!]
  \centering
  \subfloat[$t=0h$]{\includegraphics[trim=0 100 0 10, clip, width=0.88\textwidth]{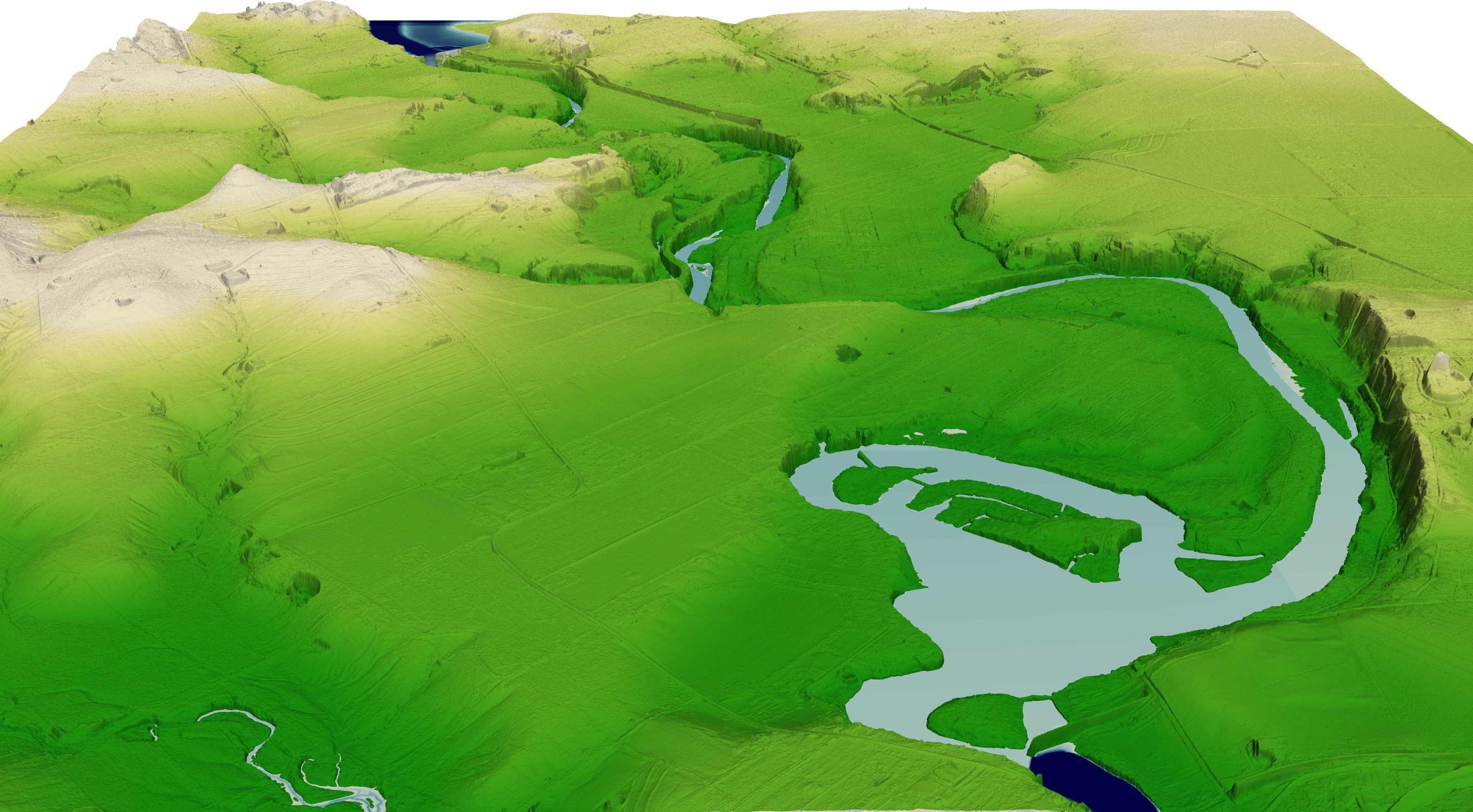}}

  \subfloat[$t=1h$]{\includegraphics[trim=0 100 0 10, clip, width=0.88\textwidth]{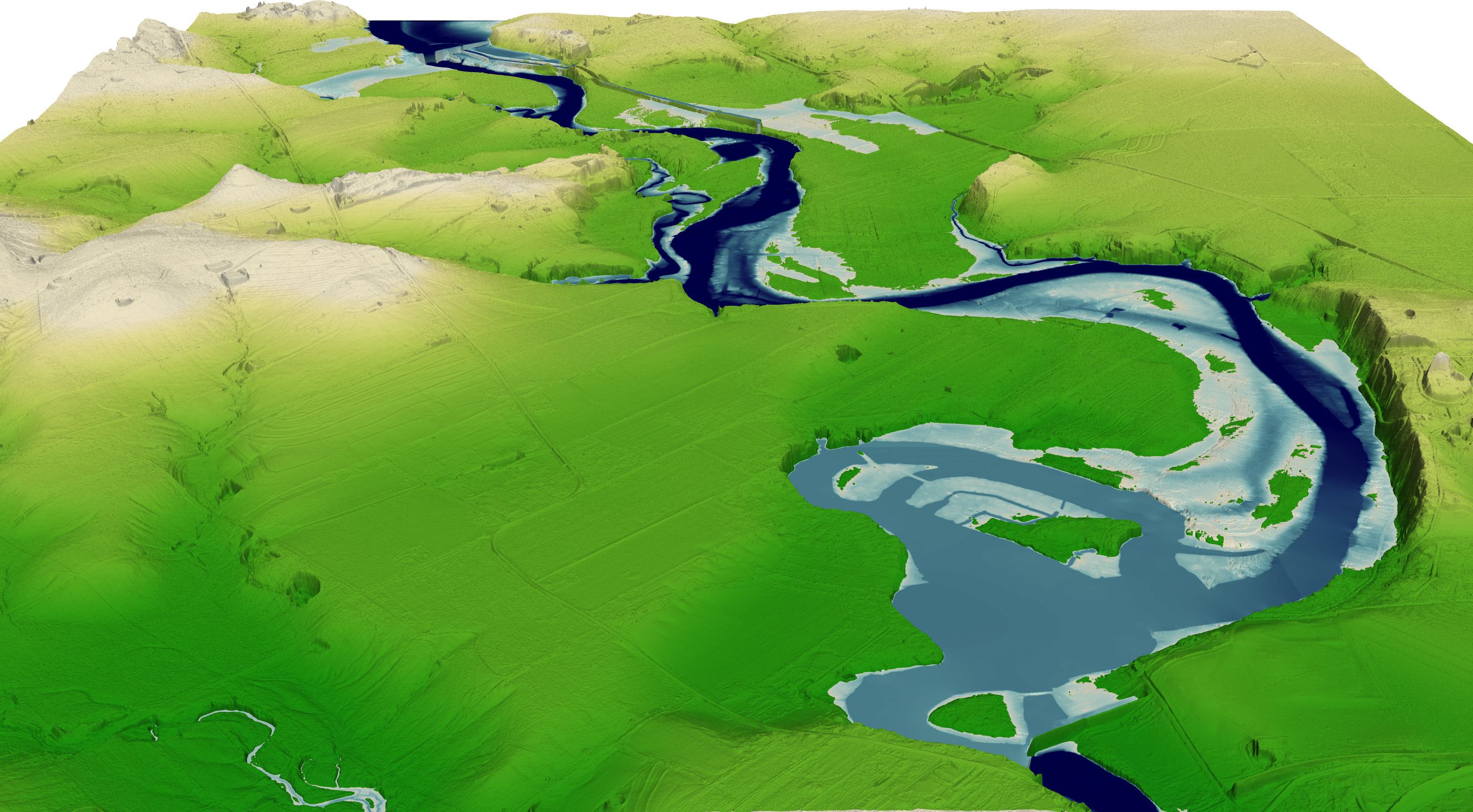}}

  \subfloat[$t=2h$]{\includegraphics[trim=0 100 0 10, clip, width=0.88\textwidth]{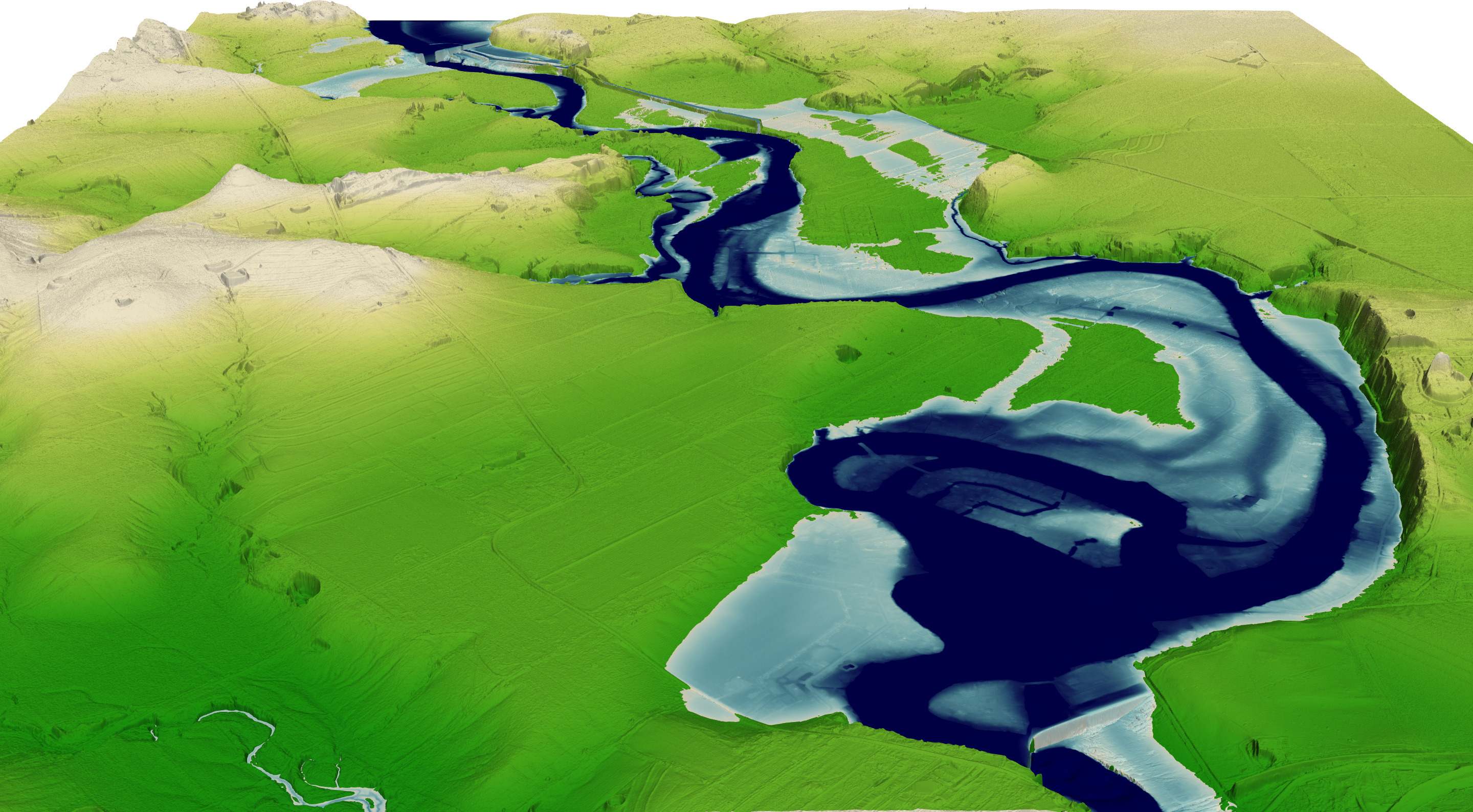}}

  \caption{
    Temporal snapshots for the high-fidelity simulation of a dam break with
    $58{,}735{,}617$ $\polQ_1$ degrees of freedom per component. The
    figures show a three dimensional rendering of the bathymetry $z(\bx)$ with
    a color scale ranging from dark green to light ocre, and the water
    surface $\waterh + z$ with a color scale ranging from dark blue (large
    $\waterh$) to light blue (small $\waterh$). The vertical direction is scaled by a
    factor 10.}
  \label{fig:dam_break}
\end{figure}
Lastly we perform a high-fidelity dam break simulation with the shallow
water equations using realistic topography data. To this end we linked
the~\texttt{ryujin} software to the \emph{Geospatial Data Abstraction
Library} (\texttt{GDAL})\footnote{\url{https://gdal.org}} in order to read
in digital elevation models (DEMs). The DEM considered here was obtained
from the \emph{United States Geological Survey 3D Elevation
Program}~\cite{USGS2021} via
\emph{OpenTopography}\footnote{\url{https://opentopography.org}} and shows
a portion of Lake Dunlap and the Guadalupe river in the state of Texas with
a spatial resolution of $\mynum{1}{m}\times\mynum{1}{m}$. We simulate the
breaking of the dam at Lake Dunlap. The simulation was performed on the
computational domain $D = [0, \mynum{7168}{m}]\times[0, \mynum{8192}{m}]$
which is the bounding box for the DEM data until final time $T=\SI{2}{h}$.
We use a Gauckler-Manning's friction source term with a roughness
coefficient of $n = \mynum{0.025}{m^{-1/3}s}$ and a gravity constant of $g
= \mynum{9.81}{ms^{-2}}$. We set up an initial water column with $\waterh + z =
\mynum{179.5}{m}$ (above sea level) for the upper basin and $\waterh + z =
\mynum{163.5}{m}$ slowly sloping down to $\mynum{161.3}{m}$ for the river
bed; both with zero initial velocity. On the northern boundary of the
domain we enforce Dirichlet conditions with $\waterh + z =\mynum{182.5}{m}$
which ensure that the upper basin is always filled. We emphasize that we
chose this flow configuration purely for demonstration purposes and that it
is not particularly realistic: we do not factor in the finite amount of
water stored in the upper basin (as it is in fact not fully simulated); we
make no attempt of creating a realistic initial configuration of the
downstream river with correct water height and stream velocity; and the DEM
does not contain bathymetry information of the river bed.

The computation was performed with $58{,}735{,}617$ $\polQ_1$ degrees of
freedom per component on 768 ranks (and 2 threads per rank) on the Whistler
cluster at Texas A\&M University using single-precision floating point
arithmetic. We performed $152{,}637$ $\text{RK}(3, 3; 1)$ steps with a
chosen CFL number of 0.9 resulting in an average time step size of
$\dt_{\text{ERK}}\approx\mynum{4.72e-02}{s}$. The total CPU time summed over
all ranks was about \mynum{9820}{h} with an average per-rank throughput of
about $\mynum{1.52}{MQ/s}$, where MQ/s stands for \emph{million
$\polQ_1$-mesh point updates per second} for a single Runge-Kutta substep
(consisting of a single forward-Euler step). We recorded a total runtime
of approximately $\mynum{6.43}{h}$ (wall time) which equates to a combined
throughput over all ranks of about $\mynum{1159.25}{MQ/s}$.

We visualize in Figure~\ref{fig:dam_break} the simulation results for
temporal snapshots at initial time $t=\mynum{0}{h}$, at $t=\mynum{1}{h}$,
and at $t=\mynum{2}{h}$. The figure shows a three dimensional rendering of
the bathymetry $z(\bx)$ with a color scale ranging from dark green to light
ocre where the bathymetry is scaled by a factor 10. Similarly, the water
surface $\waterh + z$ is overlayed with the same scaling factor and with a
color scale ranging from dark blue to light blue for large to small values
of $\waterh$.

\section{Conclusion}
\label{sec:conclusion}
In this work, we provided a high-order space and time approximation of the
shallow water equations with external sources on unstructured meshes. The
numerical method was shown to be invariant-domain preserving and
well-balanced with respect to rest states whether the shoreline is
aligned with the mesh or not. The method was also shown to be
robust with respect to external source terms.
This work will be the stepping stone for various multi-physics extensions of the shallow water
  equations like the Serre-Green-Naghdi Equations which account for dispersive water
  waves or subsurface
models such as Richard's equation.

\section*{Acknowledgments}
The authors thank Sergio Martinez for his support in providing the data
reported in~\citep{martinez_2018} as well as the \texttt{Gnuplot} files for
postprocessing the data.

\section*{Data availability statement}
A high performance implementation of the algorithms discussed in this paper
are freely available as part of the \texttt{ryujin}
project~\citep{ryujin-2021-1,ryujin-2021-3}. The source code repository is
located at \url{https://github.com/conservation-laws/ryujin}. Parameter and
configuration files for the numerical illustrations reported in
Section~\ref{sec:illustrations} are made available upon request.

\FloatBarrier
\appendix
\section{Boundary conditions}\label{sec:boundary_conditions}
In this section, we describe how the boundary conditions are enforced for the IDP explicit Runge-Kutta schemes above.
Recall that the Shallow Water Equations with flat topography (and no external source terms) are equivalent to the isentropic compressible Euler Equations when the adiabatic index $\gamma$ is 2.
Thus, the details in this section are a direct modification of the work seen in~\citep{ryujin-2021-3} where it is shown how to enforce ``reflecting'' (\ie slip or wall) and ``non-reflecting'' boundary conditions for the isentropic Euler Equations.
\subsection{Preliminaries}
Boundary conditions are enforced by post-processing the approximation at the end of each stage of the ERK-IDP algorithm.
We consider two types of boundary conditions: \textup{(i)} Reflecting conditions, also called ``slip'' or ``wall'': $\bv\SCAL\bn = 0$; \textup{(ii)} Non-reflecting conditions.
Let $\partial D\dr\subset\partial D$ be the boundary where reflecting conditions are enforced. Let $\partial D\dnr$ denote the complement of $\partial D\dr$ in $\partial D$ where non-reflecting conditions will be enforced.
Let $\calV\dr^\partial\subset\calV^\partial$ be the collection of all the boundary degrees of freedom $i$ such that $\varphi_{i|\partial D\dr}\not\equiv 0$.
Let $\calV^\partial\dr\subset\calV^\partial$ be the collection of all boundary degrees of freedom $i$ such that $\varphi_{i|\partial D\dnr}\not\equiv 0$. We define the normal vectors associated with the degrees of freedom in $\calV\dr^\partial$ and $\calV\dnr^\partial$, respectively:
\begin{equation}
    \bn_i\upr\eqq \frac{\int_{\partial D\dr}\varphi_i\bn\diff s}{\|\int_{\partial D\dr}\varphi_i\bn\diff s\|_{\ell^2}}, \qquad \bn_i\upnr\eqq \frac{\int_{\partial D\dnr}\varphi_i\bn\diff s}{\|\int_{\partial D\dnr}\varphi_i\bn\diff s\|_{\ell^2}}.
\end{equation}
In the following two sections, the symbol $\bsfU$ denotes the state obtained at the end of each stage. The post-processed state is denoted by $\bsfU\upP$.
\subsection{Reflecting boundary conditions}
Let $i\in\calV\dr^\partial$ and let $\bsfU_i = (\sfH_i, \bsfQ_i)\tr$.
Reflecting boundary conditions are enforced at the $i$-th degree of freedom by setting:
\begin{equation}
    \bsfU_i\upP \eqq (\sfH_i, \bsfQ_i - (\bsfQ_i\SCAL\bn_i^{\text{r}})\bn_i^{\text{r}})\tr.
\end{equation}
\subsection{Non-reflecting boundary conditions}
We now consider non-reflecting boundary conditions at $i\in\calV\dnr^\partial$ based on Riemann invariants.
We note that the idea of working with the Riemann invariants of the Shallow Water Equations for the use of boundary conditions is a common approach in the literature (see:~\cite[Sec.~6]{bristeau2001boundary}).
For notational simplicity, we assume the following states are at the $i$-th degree of freedom and drop the subscript notation. Here, $\bn\eqq\bn_i\upnr$.

Let $\waterh\eqq\waterh(\bsfU),\,\bsfQ\eqq\bq(\bsfU)$ and set:
\begin{equation}
    \bsfV\eqq\waterh^{-1}\bsfQ,\qquad \sfV_n\eqq\bsfV\SCAL\bn,\qquad \bsfV^\perp\eqq\bsfV-(\bsfV\SCAL\bn)\SCAL\bn,\qquad a\eqq\sqrt{g\waterh}.
\end{equation}
Assume that the topography is flat and that there are no external source terms.
Then, the characteristic variables and characteristic speeds for the one-dimensional system, $\partial_t \bu + \partial_x(\polf(\bu)\bn) = 0$, are:
\begin{equation}
    \underbrace{
        \begin{cases}
            \lambda_1(\bsfU,\bn)\eqq\sfV_n - a \\
            \sfR_1(\bsfU,\bn)\eqq\sfV_n - 2 a,
        \end{cases}
    }_{\text{multiplicity }1}
    \qquad
    \underbrace{
        \begin{cases}
            \lambda_2(\bsfU,\bn)\eqq\sfV_n \\
            \bsfV^\perp,
        \end{cases}
    }_{\text{multiplicity }d-1}
    \qquad
    \underbrace{
        \begin{cases}
            \lambda_3(\bsfU,\bn)\eqq\sfV_n + a \\
            \sfR_3(\bsfU,\bn)\eqq\sfV_n + 2 a.
        \end{cases}
    }_{\text{multiplicity }1}
\end{equation}
Note in passing there are only two Riemann invariants for the Shallow Water Equations, but we use the notation $\sfR_1$ and $\sfR_3$ so that they correspond directly to the eigenvalues.
We consider four different cases depending on the type of flow at the boundary:\\
\begin{itemize}
    \item[(i)] torrential inflow: $\hspace{8pt}\sfV_n < 0\text{ and }a<|\sfV_n| \qquad\lambda_1\leq\lambda_2\leq\lambda_3<0$,
    \item[(ii)] torrential outflow: $\hspace{4pt}0\leq\sfV_n \text{ and }a\leq|\sfV_n| \qquad0\leq\lambda_1\leq\lambda_2\leq\lambda_3$,
    \item[(ii)] fluvial inflow:$\hspace{27pt}\sfV_n < 0\text{ and }|\sfV_n|<a \qquad\lambda_1\leq\lambda_2<0\leq\lambda_3$,
    \item[(iv)] fluvial outflow: $\hspace{17pt}0\leq\sfV_n \text{ and }a<|\sfV_n| \qquad\lambda_1<0\leq\lambda_2\leq\lambda_3$.\\
\end{itemize}
Note that the nomenclature ``torrential'' is equivalent to ``supersonic'' and ``fluvial'' is equivalent to ``subsonic'' in the context of gas dynamics.
We assume that outside the domain $D$, we have at hand some Dirichlet data $\bsfU\upD\eqq(\waterh\upD, \bsfQ\upD)\tr$.
Just as in~\citep{ryujin-2021-3}, we are going to postprocess the solution $\bsfU$ so that the characteristic variables of the post-processed state $\bsfU\upP$ associated with the in-coming eigenvalues match those of the prescribed Dirichlet data $\bsfU\upD$, while leaving the out-going characteristics unchanged.
More precisely, the strategy consists of finding $\bsfU\upP$ so that the following holds:
\begin{subequations}
    \begin{align}
         & \sfR_l(\bsfU\upP) =
        \begin{cases}
            \sfR_l(\bsfU\upD) & \text{if }\lambda_l(\bsfU,\bn^{\textup{nr}})<0,    \\
            \sfR_l(\bsfU)     & \text{if }0\leq\lambda_l(\bsfU,\bn^{\textup{nr}}),
        \end{cases}
        \quad l\in\{1,3\},     \\
         & (\bsfV\upP)^\perp =
        \begin{cases}
            (\bsfV\upD)^\perp & \text{if }\lambda_2(\bsfU,\bn^{\textup{nr}})<0,    \\
            \bsfV^\perp       & \text{if }0\leq\lambda_2(\bsfU,\bn^{\textup{nr}}).
        \end{cases}
    \end{align}
\end{subequations}
We now solve the above system for each of the four flow configurations mentioned above.\\[0.5em]
\textit{Torrential inflow}: Assume that $\lambda_1(\bsfU,\bn)\leq\lambda_2(\bsfU,\bn)\leq\lambda_3(\bsfU,\bn)<0$.
Since all the characteristics are entering the computational domain, the postprocessing consists of replacing $\bsfU$ by $\bsfU\upD$:
\begin{equation}
    \bsfU\upP = \bsfU\upD.
\end{equation}
\textit{Torrential outflow}: Assume that $0\leq\lambda_1(\bsfU,\bn)\leq\lambda_2(\bsfU,\bn)\leq\lambda_3(\bsfU,\bn)$.
Since all the characteristics are exiting the computational domain, the postprocessing consists of doing nothing:
\begin{equation}
    \bsfU\upP = \bsfU.
\end{equation}
\textit{Fluvial inflow}: Assume that $\lambda_1(\bsfU,\bn)\leq\lambda_2(\bsfU,\bn)<0<\lambda_3(\bsfU,\bn)$.
Then, $\bsfU\upP$ is obtained by solving the following system:
\begin{equation}
    \sfR_1(\bsfU\upP) = \sfR_1(\bsfU\upD), \qquad (\bsfV\upP)^\perp = (\bsfV\upD)^\perp,\qquad \sfR_3(\bsfU\upP) = \sfR_3(\bsfU).
\end{equation}
This gives that $\sfV_n\upP = \frac12\left(\sfR_1(\bsfU\upD) + \sfR_3(\bsfU)\right)$ and $4a\upP = \left(\sfR_3(\bsfU) - \sfR_1(\bsfU\upD)\right) = \sfV_n + 2 a - (\sfV_n\upD - 2 a\upD)$.
Since in this flow configuration $\sfV_n + 2 a > 0$, for $a\upP$ to be positive it must be that: $\sfV_n\upD \leq 2a\upD$ which is an admissibility condition on the Dirichlet data.
Finally, the postprocessing for a fluvial inflow boundary condition consists of setting the solution $\bsfU\upP$ to:
\begin{subequations}
    \begin{align}
         & \waterh\upP = \frac1g(a\upP)^2 = \frac1g\left(\frac{\sfR_3(\bsfU) - \sfR_1(\bsfU\upD)}{4}\right)^2,                                                                        \\
         & \bsfQ\upP =\waterh\upP \times \left((\bsfV\upD)^\perp + \sfV_n\upP\bn\right), \quad \text{ with }\quad \sfV_n\upP = \frac12\left(\sfR_1(\bsfU\upD) + \sfR_3(\bsfU)\right).
    \end{align}
\end{subequations}
\textit{Fluvial outflow}: Assume that $\lambda_1(\bsfU,\bn)<0<\lambda_2(\bsfU,\bn)<\lambda_3(\bsfU,\bn)$.
Then, $\bsfU\upP$ is obtained by solving the following system:
\begin{equation}
    \sfR_1(\bsfU\upP) = \sfR_1(\bsfU\upD), \qquad (\bsfV\upP)^\perp = \bsfV^\perp,\qquad \sfR_3(\bsfU\upP) = \sfR_3(\bsfU).
\end{equation}
Notice now that only one Dirichlet condition is prescribed on the first characteristic since $\lambda_1 < 0$.
Again, we have that $\sfV_n\upP = \frac12\left(\sfR_1(\bsfU\upD) + \sfR_3(\bsfU)\right)$ and $4a\upP = \left(\sfR_3(\bsfU) - \sfR_1(\bsfU\upD)\right)$.
We also have the same admissibility condition on the Dirichlet data: $\sfV_n\upD \leq 2a\upD$ for $a\upP>0$.
Finally, the postprocessing for a fluvial outflow boundary condition consists of setting the solution $\bsfU\upP$ to:
\begin{subequations}
    \begin{align}
         & \waterh\upP = \frac1g(a\upP)^2 = \frac1g\left(\frac{\sfR_3(\bsfU) - \sfR_1(\bsfU\upD)}{4}\right)^2,                                                                  \\
         & \bsfQ\upP =\waterh\upP \times \left(\bsfV^\perp + \sfV_n\upP\bn\right), \quad \text{ with }\quad \sfV_n\upP = \frac12\left(\sfR_1(\bsfU\upD) + \sfR_3(\bsfU)\right).
    \end{align}
\end{subequations}
\begin{remark}[Conservation and admissibility] The conservation and admissibility properties of the proposed boundary conditions are described in~\citep[Sec.~4.3.3]{ryujin-2021-3}.
\end{remark}


\bibliographystyle{abbrvnat}
\bibliography{ref}


\end{document}